\documentclass[a4paper,11pt,fleqn]{article}
\usepackage[official]{eurosym}
\usepackage{amsmath}
\usepackage{amssymb, mathtools}
\usepackage{theorem}
\usepackage[usenames,dvipsnames]{pstricks}
\usepackage{euscript}
\usepackage{graphicx}
\usepackage{charter}
\usepackage{algorithm}
\usepackage{algpseudocode}
\usepackage{svg}
\usepackage{booktabs}
\usepackage{stmaryrd}
\usepackage{tikz}
\usepackage{multirow}
\usepackage{pifont}
\usepackage{array}

\ifpdf
  \DeclareGraphicsExtensions{.eps,.pdf,.png,.jpg}
\else
  \DeclareGraphicsExtensions{.eps}
\fi
\usepackage{pdflscape}

\usepackage{comment}

\input alphabet

\newcommand{\email}{}
\definecolor{labelkey}{rgb}{0,0.08,0.45}
\definecolor{refkey}{rgb}{0,0.6,0.0}
\definecolor{Brown}{rgb}{0.45,0.0,0.05}
\definecolor{dgreen}{rgb}{0.00,0.49,0.00}
\definecolor{dblue}{rgb}{0,0.08,0.75}
\RequirePackage[colorlinks,hyperindex]{hyperref} 
\hypersetup{linktocpage=true,citecolor=dblue,linkcolor=dgreen}
\usepackage{cleveref}
\PassOptionsToPackage{normalem}{ulem}
\usepackage{ulem}
\renewcommand{\leq}{\ensuremath{\leqslant}}
\renewcommand{\geq}{\ensuremath{\geqslant}}
\renewcommand{\le}{\ensuremath{\leqslant}}
\renewcommand{\ge}{\ensuremath{\geqslant}}
\newcommand{\minimize}[2]{\ensuremath{\underset{\substack{{#1}}}%
{\text{\rm minimize}}\;\;#2 }}

\newcommand{\scal}[2]{{\left\langle{{#1}\mid{#2}}\right\rangle}}

\newcommand{\Menge}[2]{\left\{{#1}~\Big|~{#2}\right\}}

\newcommand{\HH}{\ensuremath{{\mathsf H}}}

\newcommand{\Sum}{\ensuremath{\displaystyle\sum}}
\newcommand{\emp}{\ensuremath{{\varnothing}}}

\newcommand{\Id}{\ensuremath{\operatorname{I}}}

\newcommand{\RR}{\ensuremath{\mathbb{R}}}
\newcommand{\defspace}{\ensuremath{\RR^n}}
\newcommand{\RP}{\ensuremath{\left[0,+\infty\right[}}

\newcommand{\RPP}{\ensuremath{\left]0,+\infty\right[}}

\newcommand{\RX}{\ensuremath{\left]-\infty,+\infty\right]}}

\newcommand{\NN}{\ensuremath{\mathbb N}}

\newcommand{\dom}{\ensuremath{\text{\rm dom}\,}}

\newcommand{\prox}{\ensuremath{\text{\rm prox}}}
\newcommand{\proj}{\ensuremath{\text{\rm proj}}}

\newcommand{\card}{\ensuremath{\text{\rm card}}}

\newcommand{\inte}{\ensuremath{\text{\rm int}}}

\newcommand{\Ndist}{\ensuremath{\mathcal{N}}}

\newcommand{\sympart}[1]{\ensuremath{ #1^{\rm s}}}
\DeclarePairedDelimiter\norm\lVert\rVert
\DeclarePairedDelimiter\innerproduct\langle\rangle
\newcommand{\ip}[2]{\innerproduct{#1 \vert #2}}
\DeclareMathOperator*{\argmin}{argmin}

\DeclareMathOperator\Gra{Gra}
\newcommand{\Jac}{\ensuremath{\operatorname{J}}}

\newcommand{\opsymbol}{\ensuremath{F}}

\newcommand{\linmodel}{\ensuremath{{\opsymbol}^{\rm lin}}}
\newcommand{\Faff}{\ensuremath{{\opsymbol}^{\rm aff}}}

\newcommand{\op}{\ensuremath{\opsymbol_{\theta}}}
\newcommand{\monop}{\ensuremath{\op^{\rm{mon}}}}
\newcommand{\nmonop}{\ensuremath{\op^{\rm{nom}}}}
\newcommand{\linop}{\ensuremath{\op^{\rm{lin}}}}

\newcommand{\monopsym}{\widetilde{\opsymbol}^{\rm{mon}}_{\theta}}
\newcommand{\linopsym}{\widetilde{\opsymbol}^{\rm{lin}}_{\theta}}

\newcommand{\F}{\opsymbol}

\newcommand{\Flin}{\linmodel}
\newcommand{\Flinsym}{\ensuremath{\widetilde{\opsymbol}^{\rm lin}}}

\newcommand{\best}[1]{\boldsymbol{#1}}

\newcommand{\zoomedinimagebis}[2]{%
  \begin{tikzpicture}
    \node[inner sep=0pt, anchor=north west] (image) at (0,0) {\includegraphics[width=0.245\textwidth]{#1}};
    \node[inner sep=0pt, anchor=north west] (zoomed) at (0,0) {%
      \includegraphics[width=0.05\textwidth, clip, trim=#2]{#1}
    };
  \end{tikzpicture}%
}

\newtheorem{theorem}{Theorem}[section]
\newtheorem{lemma}[theorem]{Lemma}

\newtheorem{proposition}[theorem]{Proposition}
\theoremstyle{plain}{\theorembodyfont{\rmfamily}%
}
\theoremstyle{plain}{\theorembodyfont{\rmfamily}%
}
\theoremstyle{plain}{\theorembodyfont{\rmfamily}%
\newtheorem{algorithm2}[theorem]{Algorithm}}
\theoremstyle{plain}{\theorembodyfont{\rmfamily}%
\newtheorem{problem}[theorem]{Problem}}
\theoremstyle{plain}{\theorembodyfont{\rmfamily}%
}
\theoremstyle{plain}{\theorembodyfont{\rmfamily}%
\newtheorem{remark}[theorem]{Remark}}
\theoremstyle{plain}{\theorembodyfont{\rmfamily}%
\newtheorem{definition}[theorem]{Definition}}
\theoremstyle{plain}{\theorembodyfont{\rmfamily}%
}

\numberwithin{equation}{section}

\begin{document}

\title{\sffamily Learning Truly Monotone Operators\\
with Applications to Nonlinear Inverse Problems}
\author{
Younes Belkouchi$^*$, Jean-Christophe Pesquet$^*$, Audrey Repetti$^{\dagger\star}$, Hugues Talbot$^*$
\\[5mm]
\small $^*$ CentraleSupelec, Inria, Universit\'e Paris-Saclay  \\
\small $^\dagger$ School of Mathematics and Computer Sciences, Heriot-Watt University, Edinburgh, UK\\
\small $^\star$ Maxwell Institute for Mathematical Sciences, Bayes Centre, Edinburgh, UK \\
\small Emails: \email{younes.belkouchi@centralesupelec.fr}, \email{jean-christophe.pesquet@centralesupelec.fr}, \\
\small \email{a.repetti@hw.ac.uk}, and \email{hugues.talbot@centralesupelec.fr}
}
\date{}

\maketitle
\thispagestyle{empty}

\vskip 8mm

\begin{abstract}
This article introduces a novel approach to learning monotone neural networks through a newly defined penalization loss. The proposed method is particularly effective in solving classes of variational problems, specifically monotone inclusion problems, commonly encountered in image processing tasks. The Forward-Backward-Forward (FBF) algorithm is employed to address these problems, offering a solution even when the Lipschitz constant of the neural network is unknown. Notably, the FBF algorithm provides convergence guarantees under the condition that the learned operator is monotone.

Building on plug-and-play methodologies, our objective is to apply these newly learned operators to solving non-linear inverse problems. To achieve this, we initially formulate the problem as a variational inclusion problem. Subsequently, we train a monotone neural network to approximate an operator that may not inherently be monotone. Leveraging the FBF algorithm, we then show simulation examples where the non-linear inverse problem is successfully solved. 
\end{abstract}





\section{Introduction}\label{Sec:introduction}

In many image processing tasks, the objective is to solve a variational problem of the form
\begin{equation}\label{e:refmin}
\minimize{x\in C}{g(x)+h(x)}
\end{equation}
where 
$C$ is a nonempty closed convex subset of $\RR^{n}$, 
$g\colon \RR^n\to \RX$, and $h\colon \RR^n\to \RX$ are functions
which may have different mathematical properties and various interpretations.
This problem appears especially in inverse imaging problems when using Bayesian inference methods to define a maximum \textit{a posteriori} (MAP) estimate from degraded measurements \cite{calvetti_inverse_2018}. 
Basic choices for $g$ and $h$ are 
\begin{equation}\label{e:IRproblem}
(\forall x \in \RR^n)\qquad
g(x) = \|W x\|_p^p, \quad 
h(x) = \frac12 \| Hx - y\|^2,
\end{equation}
where $W\in \RR^{q\times n}$ and $H\in \RR^{m\times n}$.
Here, vector $y\in \RR^m$ corresponds to measurements corrupted by an additive white Gaussian noise, $H$ represents a linear model for the acquisition process (e.g., convolution, under-sampled Fourier, or Radon transform), $W$ is a well-chosen linear operator mapping the image $x$ to a transformed domain (e.g., wavelet transform), and $\| \cdot \|_p$ is an $\ell_p$ norm, with $p\in [1, +\infty]$.
When $p=2$, \eqref{e:IRproblem} allows us to recover a least squares formulation with a  Tikhonov regularization \cite{benning_modern_2018, calvetti_inverse_2018}, whereas  $p=1$ promotes sparse solutions to the inverse problem.
Both $H$ and $W$ may be unknown, or approximately known, and may require learning strategies to achieve high reconstruction quality.

In recent decades, proximal splitting methods \cite{combettes_proximal_2011, combettes_fixed_2021} have been extensively used to address large-scale convex or nonconvex variational problems that encompass constraints, both differentiable and nondifferentiable functions, as well as linear operators. 
Multiple classes of proximal algorithms have been designed to tackle various forms of optimization problems. Among these algorithms, one can cite first-order algorithms such as the proximal point algorithm, Douglas-Rachford algorithm \cite{combettes_douglasrachford_2007}, forward-backward (FB) algorithm (also known as proximal gradient algorithm) \cite{bauschke_convex_2017, combettes_proximal_2011}, or the forward-backward-forward (FBF) algorithm (also known as Tseng's algorithm) \cite{tseng_modified_2000}. The first two mentioned algorithms allow us to handle non-smooth functions, through their proximity operators, while the last two algorithms mix
gradient steps (i.e., explicit or forward steps) with proximal steps (i.e., implicit or backward steps). 
While both FB and FBF algorithms are capable of minimizing the sum of a differentiable function and a non-differentiable function, the former requires the differentiable function to have a Lipschitz-continuous gradient, which is not needed when using Tseng's algorithm. 

Lately, proximal algorithms have further gained attention for tackling inverse imaging problems, as their efficiency has been improved when paired with powerful neural networks (NNs). These hybrid methods consist in replacing some of the steps in the proximal algorithm by a NN that has been trained for a specific task, leading to so-called ``plug-and-play" (PnP) methods. Traditionally, the intuition behind PnPs was to see the proximity operator as a \textit{Gaussian denoiser} associated with the regularization term to compute the MAP estimate. This denoiser can be handcrafted (for example, BM3D) \cite{almeida_blind_2013} or learned (that is, NN) \cite{cohen_regularization_2020, fermanian_pnp-reg_2023, hurault_gradient_2022, hurault_proximal_2022, pesquet_learning_2021, ryu_plug-and-play_2019}. 
However, recently, a few studies studies began incorporating different types of NNs into proximal algorithms. These NNs may be trained for various tasks (e.g., inpainting \cite{tang_data-driven_2023}) or used to replace steps in the algorithm other than the proximity operator \cite{al_2022_bregman, bouman_generative_2023, hurault_gradient_2022, hurault_convergent_2023}. 
For example, the NN could be used to replace the gradient step \cite{hurault_gradient_2022}. In this work, the authors learn a denoiser that is not constrained and use Proximal Gradient Descent (PGD) as the minimization algorithm. The gradient step is modeled by using a denoising NN, which then allowed them to solve various image restoration tasks by plugging the denoiser into PGD, while having a good performance. Although the method involves a modified PGD with a backtracking step search, it does not yield the true minimum of the considered objective function since the modeled prior is non-convex. As such, converging to a critical point is guaranteed (as is the best case in most non-convex settings).
On the other hand, the same authors have proposed to learn Bregman proximal operators, which can simplify some computations in order to handle measurement corrupted with Poisson noise \cite{al_2022_bregman, hurault_convergent_2023}. 
Interestingly, the variational problem offers a more general setting, namely maximally monotone operator (MMO) theory, to enlarge the parameter space of the NN in a constrained setting, while still offering global convergence guarantees to an optimal solution.
Indeed, most proximal algorithms are grounded on MMO theory, in the sense that convergence of the iterates can be investigated for solving monotone inclusions, instead of variational problems \cite{combettes_fixed_2021}. In this context, the authors in \cite{garcia2023primal, pesquet_learning_2021, terris_enhanced_2021} proposed to learn the resolvent of a maximally monotone operator, i.e., a firmly-nonexpansive operator. 
Finally, it can be noted that deep equilibrium methods exhibit similarities with PnP algorithms \cite{bolte2024differentiating, gilton2021deep, kamilov2023plug}. They are often based on proximal fixed point equations, where some operators are replaced by NNs. These are trained based on fixed point properties and subsequently used in an iterative process. In~\cite{gilton2021deep}, the authors proposed to replace either a gradient step or a proximal step by a network. The authors in \cite{kamilov2023plug} highlighted that a deep equilibrium architecture can be obtained by running PnP iteration until convergence and using the implicit differentiation at the fixed point.

Motivated by the MMO approach developed 
in \cite{pesquet_learning_2021},  we propose to generalize~\eqref{e:refmin} to monotone inclusion problems. 
First, it can be obviously noticed that the constrained
optimization problem \eqref{e:refmin} can be rewritten as
\begin{equation}\label{e:refmingen}
\minimize{x\in \RR^n}{g(x)+h(x)+ \iota_C(x)}
\end{equation}
where $\iota_C$ denotes the indicator function of $C$
(see definition at the end of this section). 
Then, assuming that $g$ and $h$ are proper convex and 
lower semi-continuous functions, under suitable qualification conditions, $\widehat{x} \in \RR^n$ is a solution  to \eqref{e:refmingen}
if and only if it is the solution to the variation inclusion
\begin{equation}\label{e:difinc}
0 \in \partial g(\widehat{x})+\partial h(\widehat{x})
+N_{C}(\widehat{x}),
\end{equation}
where $N_C = \partial \iota_C$ is the normal cone to $C$. 
Consequently, solving \eqref{e:difinc} is a special case of the following monotone inclusion problem
where we have specified our working assumptions.
\begin{problem}\label{prob:main}
Let $C$ be a closed convex set of $\RR^n$ with a nonempty interior. Let $h\colon \RR^n\to \RR$ be a proper lower-semicontinuous convex function. Assume that $h$ is continuously differentiable on $C\subset \inte(\dom h)$.
Let $A$ be a monotone continuous operator defined on $C$. We want to 
\begin{equation}\label{e:incmain}
\text{find } \widehat{x}\in \RR^n \text{ such that }
0 \in A(\widehat{x})+\partial h(\widehat{x})+N_{C}(\widehat{x}).
\end{equation}
\end{problem}
Since \eqref{e:refmin} is an instance of Problem~\ref{prob:main} under suitable conditions, the latter provides a more general formulation, as sub-differentials of convex functions are particular cases of monotone operators. Hence, rather than being restricted to a variational model, Problem~\ref{prob:main} enlarges the problem framework to incorporate monotone operators. 
For instance, $A$ does not have to be the gradient of any function. We hence move from the Bayesian formulation of inverse problem to a monotone inclusion framework. 
Note that the main difference between \cite{pesquet_learning_2021} and this work is that the former approach is based on Minty's theorem which characterizes a maximally monotone operator through its resolvent. In \cite{pesquet_learning_2021}, the focus in terms of modeling and learning was therefore placed on the resolvent operator of $A$.
In contrast, in this work, we model and learn directly 
$A$, a ``true" monotone operator. 
To this aim, we develop a new approach based on monotone operator theory, that enables training networks with the desirable monotone property.
We note that learning directly the monotone operator has a few advantages over learning the resolvent of a maximally monotone operator (as proposed in~\cite{pesquet_learning_2021}). 
First, the ``power" of a network that approximates a monotone operator can explicitly be tuned by using a multiplicative factor in front of it. This not only enables the introduction of a regularization parameter in front of $A$ in~\eqref{e:incmain}, but also, in first-order methods, the use of a step size that will not change the limit point of the associated PnP iterations. This is not the case when learning a resolvent operator (see~\cite{garcia2023primal, pesquet_learning_2021}).
Second, with the proposed approach, the monotone operator can directly be evaluated, hence the value of~\eqref{e:incmain} can also be evaluated. Instead, when learning a resolvent operator, the operator of interest does have a theoretical expression, but in practice it would be necessary to invert the NN to evaluate it.
Third, the proposed monotone approach enables the use of varying PnP algorithms other than the standard PnP iterations. In the current work, we  propose a PnP version of Tseng's algorithm. Its use in this context is fully novel, to the best of our knowledge.
Finally, imposing monotonicity is less restrictive in the network design than imposing firm non-expansiveness, as it is needed for learning resolvent operators. Hence, the proposed monotone approach can potentially be used for solving different problems than with the resolvent operator approach. We will illustrate this last statement in our simulation section, where we will consider a nonlinear inverse problem.

Recently, monotone operator theory has found its way in the study of normalizing flows. These generative methods focus on modeling complex probability distributions using simple and known distributions (e.g. Gaussian) using NN \cite{kingma_improving_2017, rezende_variational_2016}. 
One of the major properties of these NNs is invertibility, as most models suppose that a tractable diffeomorphism exists between the mappings involving the variables of both densities, and as such, the inverse model can be computed or estimated using different methods. Invertibility was usually imposed by enforcing the non nullity of the determinant of the Jacobian of the mapping, either through re-parametrization or carefully chosen loss functions \cite{kruse_benchmarking_2021}, recent studies have focused on imposing strong monotonicity as a surrogate to invertibility \cite{ahn_invertible_2022, chaudhari_learning_2023}.
The authors of \cite{ahn_invertible_2022} impose monotonicity using the \textit{Cayley} operator associated with a given operator. They rely on the property stating that the Cayley operator of a monotone mapping is nonexpansive, hence enforces this condition through spectral normalization and newly defined 1-Lipschitz activations. 
Instead, the authors of \cite{chaudhari_learning_2023} propose two new NN architectures, namely the \textit{Cascaded Network} and the \textit{Modular Network}. Both architectures are monotone, and model the gradient field of a convex function. 
The main differences between these works and our approach is that, while the authors of \cite{ahn_invertible_2022, chaudhari_learning_2023} must impose some conditions on the architecture of their networks, our approach is agnostic to the network architecture, provided that the learning process converges to an acceptable solution. Further, since 
we aim to learn any monotone operator, which may or may not be a gradient operator, 
their proof of monotonicity can be viewed as a special case of the one provided later in this paper.

Lastly, we choose to tackle  an original nonlinear inverse problem in the context of image restoration, as an application of Problem~\ref{prob:main}, which constitutes a different problem from normalizing flows.
Note that, although many works have been dedicated to linear inverse problems, the investigation of nonlinear degradation models is more scarce \cite{colibazzi2023deep}.

In this article, our main focus lies on scenarios in which the operator $A$ is unknown. To address this challenge, our methodological contribution to the field is two-fold. First, we tailor an existing algorithm to effectively tackle the inclusion Problem~\ref{prob:main}. Second, we introduce a comprehensive framework that harnesses NNs to learn and replace the monotone operator $A$, leveraging their universal approximation capabilities. Our contributions encompass both algorithmic adaptation and the establishment of a novel approach for monotone operator learning through NNs. 
In particular, we introduce a new penalization function that can be used with usual optimizers (e.g., SGD, Adam, etc.) during the training process.
To demonstrate the potential of our approach, we define and analyze a non-linear inverse problem which is modeled as a monotone inclusion problem, and we show that we are able to solve it using a learned monotone operator.

This article is structured as follows. In Section~\ref{Sec:problemstatement}, we introduce Tseng's algorithm and adapt it to solve Problem~\ref{prob:main}. In Section~\ref{Sec:proposedapproach}, we propose a PnP framework based on Tseng's iterations. In this section, we also relate monotonicity of NNs to the study of their Jacobian. We introduce the associated penalization that will be used during the training process to learn monotone NNs. 
In Section~\ref{Sec:simulations},
we 
formulate a non-linear inverse problem as a variational inclusion problem, and solve it using the presented tools. Lastly, conclusions are drawn in Section~\ref{Sec:conclusions}.

\smallskip
\paragraph{Notation} 
$(\defspace, \norm{ \cdot })$ is a Euclidean space, where $ \norm{ \cdot }$ is the $\ell_2$ norm and  $\ip{\cdot}{\cdot}$ is the associated inner product.
Let $S\subset \RR^n$, the interior of $S$ is denoted by $\inte S$.
An operator $A: \defspace \rightrightarrows 2^{\defspace}$ is a set-valued map if it maps every $x \in \defspace$ to a subset $A(x) \subset \defspace$. 
$A$ is single valued if, for every $x \in \defspace$,  $\card A(x) = 1$, in which case we consider that $A$ 
can be identified with an application from
$\defspace$ to $\defspace$.
The \textit{graph} of the set valued operator $A$ is defined as $\Gra A = \{ (x, u)\in (\defspace)^2 \mid u \in A(x)\}$. 
The reflected operator of $A$ is defined as $R_A = 2A - \Id$.

The Moreau subdifferential of a convex function $f$ is a set valued operator denoted as $\partial f$. 
If $f$ is differentiable, then $\partial f$ is single valued, in which case $\nabla f$ refers to its gradient: $(\forall x \in \defspace)$ $\partial f(x) = \{ \nabla f(x) \}$. In addition, if $f$ is proper and lower-semicontinuous, the proximity operator 
of $f$ associates to every $x\in \defspace$ the unique point
$p\in \defspace$ such that $x-p \in \partial f(p)$.
 The indicator function $\iota_C$ of a subset $C$ of $\defspace$ is defined as:
$(\forall x \in \defspace)$
$\iota_C(x) = 0$ if $x \in C$, and $\iota_C(x) = +\infty$ otherwise.
The normal cone to a set $C \subset \defspace$ at $x\in \defspace$ is defined as $N_C(x) = \{ u \in \defspace \mid (\forall y \in \defspace) \scal{y - x}{u} \leq 0\}$ if $x \in C$, and $N_C(x) = \emp$ otherwise.
In particular, if $C$ is a nonempty closed convex set,
$\partial \iota_C = N_C$ and
$\prox_{\iota_C}= \proj_C$ is the projector onto $C$.

Let $A\colon \defspace \rightrightarrows 2^{\defspace}$. Then $A$ is a \textit{monotone} (resp. \textit{strictly monotone}) operator if $(\forall (x,u) \in \Gra A)(\forall (y,v) \in \Gra A) \quad \scal{u - v}{x - y} \geq 0$ (resp. $ \scal{u - v}{x - y} > 0)$ if $x\neq y$). 
Furthermore, $A$ is a $\beta$-\textit{strongly monotone} operator, with constant $\beta>0$, if $(\forall (x,u)\in \Gra A)(\forall (y,v) \in \Gra A) \quad \scal{u - v}{x - y} \geq \beta \|x-y\|^2 $. 
As a limit case, 
a monotone operator can be considered as $0$-strongly monotone. We say that $A$ is \textit{maximally} monotone if there exists no monotone operator $B\neq A$ such that $\Gra A \subset \Gra B$. When $f \colon \defspace \to ]-\infty, +\infty]$ is convex, proper, and lower semi-continuous, then $\partial f$ is maximally monotone. A is $\beta$-cocoercive with constant $\beta>0$, if $(\forall (x,u)\in \Gra A)(\forall (y,v) \in \Gra A) \quad \scal{u - v}{x - y} \geq \beta \|u-v\|^2 $.

Let $T\colon \defspace \to \defspace$ be a single valued operator, and $x \in \RR^n$. If $T$ is Fr\'echet differentiable, we denote by $\Jac_T(x) \in \RR^{n\times n}$ the Jacobian of $T$ at $x$ which is defined as the matrix that verifies
$\lim_{\|h\| \to 0, h\neq 0} \frac{\|T(x+h) - T(x) - \Jac_T(x)h\|}{\|h\|} = 0$. If $T$ is Fr\'echet differentiable, then $T$ is G\^ateaux differentiable and, for every $h\in\defspace$ $\Jac_T(x)h = \lim_{t \to 0,t\neq 0} \frac{T(x+th)-T(x)}{t}$. 
Further, we define the symmetric part of its Jacobian operator as $\Jac^{\rm s}_T=(\Jac_T + \Jac_T^\top)/2$. 

Let $M \in \RR^{n \times n}$ be a symmetric matrix. We denote by $\lambda_{\min}(M)$ the smallest eigenvalue of $M$, and by $\overline{\lambda}_{\max}(M)$ the 
maximum absolute value of the eigenvalues of $M$.
Let $(M_1, M_2) \in (\RR^{n \times n})^2$. The Loewner order $M_1 \succ M_2$ (resp. $M_1 \succeq M_2$) is defined as, for every $u\in \defspace\setminus\{0\}$, $u^\top M_1 u > u^\top M_2 u$ (resp. $u^\top M_1 u \geq u^\top M_2 u$).
Then $M$ is positive definite (resp. positive semidefinite) if and only if $M \succ 0$ (resp. $ M \succeq 0$). An alternative definition is that $M$ is positive definite (resp. positive semidefinite) if and only if $\lambda_{\min}(M)>0$ (resp. $\lambda_{\min}(M) \geq 0$).

For further background on convex optimization and MMO theory, we refer the reader to \cite{bauschke_convex_2017}.


\section{Tseng's algorithms for monotone inclusion problems}\label{Sec:problemstatement}
\label{Sec:Tseng}

The algorithm we propose in this work is grounded on the original work by Tseng in \cite{tseng_modified_2000}. 
In this section, we first recall the relevant background, and then give our 
modified version of Tseng's algorithm for solving \eqref{e:incmain}.

\subsection{Forward-Backward-Forward strategy}
\label{Ssec:Tseng:simple}

The following monotone inclusion problem is considered
in \cite{tseng_modified_2000}.
\begin{problem}\label{prob:difincT}
Let $C$ be a nonempty closed convex subset of $\RR^n$. 
Let $f\colon \RR^n\to ]-\infty, + \infty]$ be a proper, lower-semicontinuous, convex function, and let $B$ be a maximally monotone operator, continuous on $\dom \partial f \subset C$. We want to 
\begin{equation}\label{e:difincT}
\text{find } \widehat{x}\in C
\text{ such that }
0 \in B(\widehat{x})+\partial f(\widehat{x}).
\end{equation}
\end{problem}

Tseng proposed to solve Problem~\ref{prob:difincT} using the following algorithm.
\begin{algorithm2}\label{a:FBFgen}
Let $x_{0}\in \dom B$ and $(\gamma_k)_{k\in \NN}$ be a sequence in $[0,+\infty[$.
\begin{equation}\label{e:TsengalgoGeneral}
\begin{array}{l}
\text{For } k = 0, 1, \ldots, \\
\left\lfloor
\begin{array}{l}
b_k = B(x_k) \\
z_{k} = \prox_{\gamma_k f}(x_k-\gamma_k b_k) \\
x_{k+1} = \proj_{C}\big(z_{k}-\gamma_{k} \big(B (z_{k})- b_k\big)\big).
\end{array}
\right.
\end{array}
\end{equation}
\end{algorithm2}
This method, sometimes also called forward-backward-forward (FBF) algorithm, is reminiscent of extragradient methods. The step-size can be determined using Armijo-Goldestein rule, defined below.
\begin{definition}[Armijo-Goldstein rule]\label{def:armijo}
Let $\sigma \in \RPP$ and $(\beta,\theta) \in ]0,1[^2$.
Let $(x_k)_{k\in \NN}$ be a sequence generated by Algorithm~\ref{a:FBFgen}.
At every iteration $k\in \NN$,
the stepsize $\gamma_{k}>0$ is chosen such that $\gamma_k = \sigma \beta^{i_k}$ where 
\begin{equation}
    i_k = \inf \left\{ i \in \NN \, \left| \,
\begin{array}{l}
\gamma =\sigma \beta^i\\
 \gamma \| B( z_k ) -B(x_{k})\| \le \theta \| z_k - x_{k}\|\\
 z_k =\prox_{\gamma f} \big( x_k-\gamma B(x_k) \big)
\end{array}
\right.
\right\}.
\end{equation}

\end{definition}

We have the following convergence result \cite[Theorem 3.4]{tseng_modified_2000}.
\begin{proposition} \label{prop:convFBF}
Consider Problem~\ref{prob:difincT}.
Let $(x_{k})_{k\in \NN}$ be a sequence generated by Algorithm~\ref{a:FBFgen}. 
Assume that 
\begin{enumerate}
\item\label{prop:convFBFi} 
there exists a solution to the problem;
\item\label{prop:convFBFiii} $B+\partial f$ is maximally monotone; 
\item\label{prop:convFBFiv} for every $x\in C$,
\begin{equation}\label{e:defvarphi}
\varrho(x) = \inf_{w\in B(x)+\partial f(x)} \|w\|
\end{equation}
 is locally bounded;
\item\label{prop:convFBFv} 
$(\gamma_k)_{k\in \NN}$ satisfies the rule given by Definition~\ref{def:armijo}.

\end{enumerate}
Then $(x_{k})_{k\in \NN}$ converges to a solution to 
the problem.
\end{proposition}

A strength of this algorithm is that it does not require the operator $B$ to be 
cocoercive as in the standard FB algorithm.
In the particular case when $B$ is $\beta$-Lipschitzian with $\beta \in \RPP$, the stepsize $\gamma_{k}$ can be chosen such that $\inf_{k\in \NN}\gamma_{k}>0$ and $\sup_{k\in \NN}\gamma_k < 1/\beta$,
which allows the use of a constant stepsize \cite{tseng_modified_2000}. Unfortunately, this implies estimating $\beta$, which may be difficult.
For example, it is known that a standard neural network with nonexpansive activation functions is Lipschitz continuous, but 
its Lipschitz constant itself is unknown and needs to be estimated. 
In this context, computing a tight estimate of
the Lipschitz constant for such a network is generally an NP-hard problem \cite{combettes_lipschitz_2020}.

\subsection{An instance of Tseng's algorithm}
\label{Ssec:Tseng:mono}

In this section, we propose to use the results presented in Section~\ref{Ssec:Tseng:simple} to derive an algorithm for solving Problem~\ref{prob:main}.
The proposed algorithm is given below.
\begin{algorithm2}\label{a:FBF}
Let $x_{0}\in C$ and $(\gamma_k)_{k\in \NN}$ be a sequence in $]0, +\infty[$.
\begin{equation}\label{e:TsengalgoBase}
\begin{array}{l}
\text{For } k = 0, 1, \ldots, \\
\left\lfloor
\begin{array}{l}
a_{k} = A (x_{k})+\nabla h(x_{k})\\
z_{k} = \proj_{C}(x_{k}-\gamma_{k} a_{k})\\
x_{k+1}= \proj_{C}\big(z_{k}-\gamma_{k} (A(z_{k})+\nabla h(z_{k})-a_{k})\big).
\end{array}
\right.
\end{array}
\end{equation}
\end{algorithm2}
The convergence of Algorithm~\ref{a:FBF} can then be deduced from Proposition~\ref{prop:convFBF}, which yields the following result.
\begin{proposition}\label{p:conv}
Consider Problem~\ref{prob:main}.
Let $(x_k)_{k\in \NN}$ be generated by Algorithm~\ref{a:FBF}.
Assume that 
\begin{enumerate}
\item\label{p:convi} there exists a solution to Problem~\ref{prob:main};

\item\label{p:convii} for every $k\in \NN$, $\gamma_k$ satisfies the Armijo rule given in Definition~\ref{def:armijo}, for $B=A + \nabla h$, and $f = \iota_C$.
\end{enumerate}
Then $(x_{k})_{k\in \NN}$ converges to a solution to Problem~\ref{prob:main}, belonging to $\dom A$.
\end{proposition}

\begin{proof}
According to \cite[Theorem 20.21]{bauschke_convex_2017}, there exists a maximally monotone extension $\widetilde{A}$ of $A$ on $\RR^n$, and \eqref{e:incmain} has the form of \eqref{e:difincT} with
$f= \iota_{C}$ and $B= \widetilde{A}+\partial h$. 
Then $f$ is proper, lower semi-continuous, and convex. Further, since $h$ is also proper, lower semi-continuous, and convex, then $\partial h$ is maximally monotone. 
Since $\widetilde{A}$ is maximally monotone and $\dom \widetilde{A} \cap \inte(\dom \partial h) \supset \dom A \cap \inte(\dom \partial h)
= C \cap \inte(\dom \partial h) \supset C \cap \inte(\dom h) = C \neq \emp$, then $B$ is maximally monotone \cite[Corollary 25.5(ii)]{bauschke_convex_2017}. 
In addition, $B=\nabla h + \widetilde{A} = \nabla h +A$ is continuous on $\dom \partial f = C$. 

We have $\dom B = \dom \widetilde{A}\cap \dom(\partial h) \supset \dom A \cap \inte(\dom h) = C$
and thus $\inte(\dom B) \cap \dom \partial f \supset \inte(C) \neq \emp$.
Consequently, Assumption \ref{prop:convFBFiii} in Proposition~\ref{prop:convFBF} holds \cite[Corollary 25.5(ii)]{bauschke_convex_2017}. 

Let $\varrho$ be defined on $C$ by \eqref{e:defvarphi}.
We have 
\begin{align}
(\forall x \in C) \quad \varrho(x) &=\inf_{t\in N_{C}(x)} \|B(x)+t\|\nonumber\\
& \le \|B(x)\| + \inf_{t\in N_{C}(x)} \|t\|\nonumber\\
& \le \|B(x)\| + \|x-\proj_{C}x\| \nonumber\\
& = \|B(x)\|.
\end{align}
Since we have seen that $B$ is continuous on $C$, $\varrho$ is locally bounded on $C$
and Assumption \ref{prop:convFBFiv} in Proposition~\ref{prop:convFBF}
is satisfied.

The convergence result thus follows from Proposition~\ref{prop:convFBF} by noticing 
that\linebreak $\prox_{\gamma_{k} f} = \proj_{C}$.
\end{proof}


\section{Proposed method using a monotone NN} 
\label{Sec:proposedapproach}

The objective of this work is to develop a PnP version of the proposed Tseng's algorithm described in Section~\ref{Ssec:Tseng:mono}. 
To this aim, we will approximate the operator $A$ in Problem~\ref{prob:main} by a monotone neural network $\op$, with learned parameters $\theta $ belonging to some set $\Theta$. Then, the modified form of Algorithm~\ref{a:FBF}
is given below.
\begin{algorithm2}\label{a:FBF-pnp}
Let $x_{0}\in C$ and $(\gamma_k)_{k\in \NN}$ be a sequence in $]0, +\infty[$.
\begin{equation}\label{e:Tsengalgo-NN}
\begin{array}{l}
\text{For } k = 0, 1, \ldots, \\
\left\lfloor
\begin{array}{l}
a_{k} = \op (x_{k})+\nabla h(x_{k})\\
z_{k} = \proj_{C}(x_{k}-\gamma_{k} a_{k})\\
x_{k+1}= \proj_{C}\big(z_{k}-\gamma_{k} (\op(z_{k})+\nabla h(z_{k})-a_{k})\big).
\end{array}
\right.
\end{array}
\end{equation}
\end{algorithm2}
In this section, we will describe how to build such a monotone neural network.

\subsection{Properties of differentiable monotone operators}

As described \linebreak in Proposition~\ref{p:conv}, to ensure convergence of a sequence $(x_k)_{k\in \NN}$ generated by Algorithm~\ref{a:FBF-pnp}, the NN $\op$ must be monotone and continuous on $C$. 
As most standard neural networks are continuous, especially those that use non-expansive activation function \cite{pesquet_learning_2021, combettes_lipschitz_2020}, we only need to focus on ensuring monotonicity of $\op$. 

Before providing a characterization of (strongly) monotone operators, let us state the following algebraic property.
The proof is skipped
due to its simplicity. 

\begin{lemma}\label{lem:eigenvalues}
    Let $M \in \RR^{N \times N}$ be a symmetric matrix,
    let $\rho \in [\overline{\lambda}_{\max}(M),+\infty[$,  and let $M' = \rho \Id - M $.
    Then
    \begin{equation}\label{e:lambdamin}
        \lambda_{\min}(M)
        = \rho
        - \overline{\lambda}_{\max}(M').
    \end{equation}
\end{lemma}

The following conditions will be leveraged later to enforce monotonicity of the network during its training.

\begin{proposition}\label{prop:jacobianmonotony}
    Let $T \colon \RR^n \to \RR^n$ be G\^{a}teaux 
    differentiable, and let, for every $x \in \RR^n$, $\Jac_T(x) $ be the Jacobian of $T$ evaluated at $x$. Let $R_T$ be the reflected operator of $T$, let $\beta \in \RP$, and let $\rho \in [\overline{\lambda}_{\max} \big( \sympart{\Jac_{R_T}}(x) \big) ,+\infty[$.
    The following properties are equivalent:
\begin{enumerate}
    \item \label{penproof1} 
    $T$ is $\beta$-strongly monotone;
    \item \label{penproof2}
    For every $ x \in \RR^n$,  $\sympart{\Jac_T}(x) \succeq \beta \Id$; 
        \item \label{prop:res-reflected:ii}
    For every $ x \in \defspace$,  $\sympart{J_{R_T}}(x) \succeq (2\beta - 1) \Id$;
            \item \label{cor:mono-sym:iii}
    For every $x \in \RR^n$, 
    \begin{equation} \label{cor:mono-sym:def-lmin}
    \rho - \overline{\lambda}_{\max} \Big( \rho \Id - \sympart{\Jac_{R_T}}(x) \Big) 
        \geq 2\beta-1.
    \end{equation}

\end{enumerate} 
In addition, if $\beta > 0$, then $T$ is invertible.
\end{proposition}

\begin{proof}
We provide a short proof for completeness.
We first show that \ref{penproof1} $\Rightarrow$ \ref{penproof2}. 
By definition \cite[Definition 22.1]{bauschke_convex_2017} of a $\beta$-strongly monotone operator, we have, for every $(x,h) \in (\RR^n)^2$ and $\alpha \in ]0,+\infty[$,
    \begin{equation*}
        \scal{h}{T(x+h) - T(x)} \geq \beta \|h\|^2   
        \Leftrightarrow \quad
        \scal{h}{T(x+ \alpha h) - T(x)} \geq \beta \alpha \|h\|^2 .
    \end{equation*}
    As $T$ is G\^ateaux differentiable, we deduce that
    \begin{align}
        \lim_{\substack{\alpha \to 0 \\ \alpha > 0}}  
        \scal{h}{\frac{T(x+\alpha h) - T(x)}{\alpha}} \geq \beta \|h\|^2  
        &   \quad \Leftrightarrow \quad
        \scal{h}{\Jac_T(x)h} \geq \beta \|h\|^2 \label{prop:pr:ineqJ} \\ 
        &   \quad \Leftrightarrow \quad
        \scal{h}{\sympart{\Jac_T}(x)h} \geq \beta \|h\|^2. \nonumber
    \end{align}
Hence, for every $x \in \RR^n$, $\sympart{\Jac_T}(x) \succeq \beta \Id$.

\smallskip

We now show that \ref{penproof2} $\Rightarrow$ \ref{penproof1}. 
Let $(x, h)\in (\RR^n)^2$, and let $\phi\colon [0,+\infty[ \to \RR  $ be defined as 
    \begin{equation*}
        (\forall \alpha \in [0, +\infty[)\quad
        \phi(\alpha) = \scal{h}{T(x+\alpha h) - T(x)} - \beta \alpha \|h\|^2.
    \end{equation*}
    We notice that $\phi$ is differentiable on $[0,+\infty[$ and its derivative $\phi'$ is such that $\phi(0) = 0$ and that, for every $\alpha \in [0, +\infty[$, $\phi'(\alpha) = \scal{h}{\Jac_T(x) h} - \beta \|h\|^2$.
    According to \ref{penproof2} and \eqref{prop:pr:ineqJ}, we have, for every $ \alpha \in [0, +\infty[$, $\phi'(\alpha) \geq 0$. 
    Thus $\phi$ is an increasing function and 
    \begin{equation*}
        (\forall \alpha \in [0, +\infty[) \quad 
         \scal{h}{T(x+\alpha h) - T(x)} - \beta \alpha  \|h\|^2 = \phi(\alpha) \geq \phi(0) = 0.
    \end{equation*}
    Hence $T$ is $\beta$-strongly monotone. 

By using the definition of positive (semi) definiteness,
\ref{penproof2} is equivalent to \linebreak$\lambda_{\min}\big( \sympart{\Jac_T}(x) \big) \ge \beta$, for every $x \in \RR^n$.

The equivalence with \ref{cor:mono-sym:iii} is a direct consequence of Lemma~\ref{lem:eigenvalues}.

In addition, we have
    \begin{align}
    \text{\ref{penproof2}}
        & \quad \Leftrightarrow \quad
        (\forall u \in \defspace) \quad  
        u^\top \sympart{J_T}(x) u \geq \beta \norm{u}^2 \nonumber\\
        &\quad \Leftrightarrow \quad
        (\forall u \in \defspace) \quad  
          u^\top (2\sympart{J_T}(x) - \Id) u \geq (2 \beta - 1) \norm{u}^2 \nonumber 
    \end{align}    
    Since $2\sympart{J_T}(x) - \Id = \sympart{J_{R_T}}$, we deduce that \ref{penproof2} is equivalent to \ref{prop:res-reflected:ii}.

    By using the definition of positive (semi) definiteness,
\ref{prop:res-reflected:ii} is equivalent to
\begin{equation*}
(\forall x \in \RR^n)\quad 
\lambda_{\min}\big( \sympart{\Jac_{R_T}}(x) \big) \ge 2\beta-1.
\end{equation*}
The equivalence with \ref{cor:mono-sym:iii} is a direct consequence of Lemma~\ref{lem:eigenvalues}.

        The last statement straightforwarly follows from \cite[Corollary~20.28]{bauschke_convex_2017} and \cite[Proposition~22.11]{bauschke_convex_2017} (see also \cite{ahn_invertible_2022}). Since $T$ is strongly monotone, it is strictly monotone, hence injective. Moreover, $T$ is single valued, monotone and continuous and is thus maximally monotone. Lastly, a maximally strongly monotone operator is surjective.
\end{proof}

\begin{remark}\  \label{re:monot}
\begin{enumerate}
    \item The previous proposition allows us to build more general operators than gradients of convex functions.
For example,  let $f \colon \HH \to \RR$ be $\beta$-strongly convex, and twice differentiable. 
    Let $S\colon \HH \to \HH$ be Fr\'echet differentiable and anti-symmetric, i.e., for every $x \in \HH$, $S(x) = -S(x)^\top$.
    Then, as a direct consequence of Proposition~\ref{prop:jacobianmonotony}, $T = \nabla f + S$ is $\beta$-strongly monotone.
    \item \label{re:monotii} 
    Note that NNs using differentiable activation functions (e.g., sigmoids or softmax) are Fr\'echet differentiable. Hence the G\^{a}teaux differentiability assumption required in Proposition~\ref{prop:jacobianmonotony} is fulfilled by such NNs. As highlighted in \cite{bolte_conservative_2020}, properties satisfied for such activation functions generalize to standard non differentiable ones (e.g., ReLU).

    \item \label{re:monotiii} Proposition~\ref{prop:jacobianmonotony}\ref{cor:mono-sym:iii} enables to characterize a differentiable (strongly) monotone operator 
    through the maximum magnitude eigenvalue of some matrices, which can be easily computed  by power iteration. Further details will be provided in the next section
    to train a monotone NN. 
    \item The sufficient strong monotonicity condition for the invertibility of $T$ is restrictive.
    In \cite{ahn_invertible_2022}, a sufficient condition for the strong monotonicity based on the Lipschitzianity  of the Cayley operator associated to $T$ is employed.  Nevertheless, these conditions are not necessary to guarantee the convergence of Algorithm~\ref{a:FBF-pnp}.
\end{enumerate}
\end{remark}

\subsection{Proposed regularization approach for training monotone NNs}\label{ssec:proposedapproach:reg}
We will now leverage the results given in the previous sections to design a method for training monotone NNs. 
In a nutshell, we will make the assumption that, for every $\theta\in \Theta$, $\op$ is Fr\'echet differentiable (see Remark~\ref{re:monot}\ref{re:monotii}) and 
resort to the power iteration method to control the eigenvalues
of the NN, in accordance with Remark~\ref{re:monot}\ref{re:monotiii}.

Subsequently, we will propose an optimization strategy to train monotone NNs. The basic principle for training a NN is to learn its parameters on a specific finite dataset consisting of pairs of input/groundtruth data denoted by $\mathbb{D}_{\rm train}\subset \defspace\times  \defspace$.  
Then the objective is to
\begin{equation}    \label{min:loss-simple}
    \minimize{\theta \in \Theta} \Sum_{(x, y) \in \mathbb{D}_{\rm train}} \mathcal{L}(\op(x), y),
\end{equation}
where $\mathcal{L}\colon \defspace\times  \defspace \to \mathbb{R}$ is a loss function. 
Such a problem can then be solved efficiently using 
stochastic first-order methods
such as Stochastic Gradient Descent, Adagrad, or Adam \cite{schmidt_descending_2021}.
In this paper, we are interested in training NNs that are monotone. In this case, \eqref{min:loss-simple} becomes a constrained optimization problem:
\begin{equation}    \label{min:loss-simple-const}
    \minimize{\theta \in \Theta} \Sum_{(x, y) \in \mathbb{D}_{\rm train}} \mathcal{L}(\op(x), y)
    \quad \text{ such that } \op \text{ is monotone}.
\end{equation}

Standard gradient-based optimization algorithms are not inherently tailored for solving problems with constraints. While it is possible to pair some of these algorithms with projection techniques (as in \cite{neacsu_emg_2024, terris_building_2020}), there are cases when identifying an appropriate projection method can be challenging. Moreover, utilizing projection methods may introduce convergence issues, particularly in scenarios involving non-convex constraints or projection operators with no closed form.

In the literature, we can find two deep learning frameworks ensuring constraints to be satisfied by NNs: either using a specific architecture definition \cite{chaudhari_learning_2023, daniels_monotone_2010}, or enforcing the constraint iteratively during the training procedure, possibly by adopting a penalized approach rather than a constrained one \cite{ahn_invertible_2022, pesquet_learning_2021, terris_building_2020}.
The first method consists in modifying the architecture to ensure that the solution always satisfies the desired property (for example the output of a ReLU function is guaranteed to be nonnegative). The second method enables to constrain any NN architecture at the expense of a higher training complexity.
In this work, we will adopt the second approach, and propose a training procedure imposing monotonicity to the NN independently of its architecture by solving a penalized problem.

First, according to Proposition~\ref{prop:jacobianmonotony}\ref{prop:res-reflected:ii}, we notice that \eqref{min:loss-simple-const} can be rewritten as
\begin{equation}    \label{min:loss-const}
    \minimize{\theta \in \Theta} \Sum_{(x, y) \in \mathbb{D}_{\rm train}} \mathcal{L} \big( \op(x) , y \big) 
    \quad \text{ such that } 
    (\forall x \in \defspace)\quad \lambda_{\text{min}} \big( \sympart{\Jac_{R_{\op}}}(x) \big) \geq -1. 
\end{equation}
Equivalently, we could use the constraint $(\forall x \in \defspace)$ $\lambda_{\text{min}} \big( \sympart{\Jac_{\op}}(x) \big)\ge 0$, but we observed that the former formulation leads numerically to a more stable training. 
As suggested earlier, we introduce an external penalization approach for solving \eqref{min:loss-const} \cite[Section 13.1]{luenberger1984linear}. The proposed penalty function $\mathcal{P}$ is then given by
\begin{equation}\label{e:penalizationraw}
    (\forall \theta \in \Theta)(\forall {x} \in \defspace) \quad 
    \mathcal{P}(\theta, {x}) = -\min \big\{ 1+\lambda_{\text{min}}(\sympart{\Jac_{R_{\op}}}({x})), \epsilon \big\}.    
\end{equation}
In the above definition, $\epsilon\in \RPP$ is a \textit{severity} parameter used to control the enforcement of the constraint.
Further, imposing
the constraint for all possible images in $\defspace$ is not tractable in practice. Instead, we impose it on a subset of images $\mathbb{D}_{\rm penal} \subset \defspace$ similar to the type of images on which the network will be applied. Specifically, $\mathbb{D}_{\rm penal}$ is defined as
\begin{equation}
    \mathbb{D}_{\rm penal} 
    = \Menge{\widetilde{x}\in \defspace }{
    (\exists  (x,y) \in \mathbb{D}_{\rm train})(\exists \nu\in [0, 1])\, 
    \widetilde{x} = \nu x + (1 - \nu) y
    }.
\end{equation}
Then the intuition behind definining this set is to use a modified version of $\mathbb D_{\rm train}$.
First, for computation efficiency, we only take a subset of $\mathbb D_{\rm{train}}$. This is motivated by the fact that in practice computing the Jacobian (see \eqref{e:penalizationraw}) on a full batch to select the smallest eigenvalue would be very costly.
Second, for robustness, instead of working on either noiseless (ground truth) or fully noisy images, we build images that are a convex combination of these two images. In practice, we generate points in $\mathbb{D}_{\rm penal}$ by 
choosing $\nu$ as a realization of a random variable with uniform distribution on $[0,1]$.

Hence, the resulting penalized problem we introduce for training monotone NN is given by
\begin{equation}\label{e:min:loss-pen-gen}
        \minimize{\theta \in \Theta} \Sum_{(x, y) \in \mathbb{D}_{\rm train}} \mathcal{L}(\op(x), y) + \xi \sum_{\widetilde{x} \in \mathbb{D}_{\rm penal}} \mathcal{P}(\theta, \widetilde{x}),
\end{equation}
where $\xi \in \RPP$ is the penalization factor.

\subsection{Penalized training implementation}
\label{Ssec:proposedapproach:training}

This section aims to present the practical procedure we employ for solving \eqref{e:min:loss-pen-gen}, so as to train a monotone neural network.
We will thus provide pseudo-codes usable with deep learning frameworks such as Pytorch. In the following, we first describe how we handle the penalization function $\mathcal{P}$. We then explain how to leverage this to solve our constrained optimization problem.

\smallskip

\begin{algorithm}[b!]
\caption{Computation of $\lambda_{\min} \big( \sympart{\Jac_{R_{\op}}}(\widetilde{x}) \big)$ 
}
\label{a:computingsmallestev}
\begin{algorithmic}
\State \textbf{Input:}
\begin{itemize}
    \item $R_{\op}, \widetilde{x}\in \mathbb{D}_{\rm penal}$ \Comment{\textit{Neural network model and data}}
    \item $N_{\rm iter}$ \Comment{\textit{Parameter}}
\end{itemize} 

\vspace{0.4cm}
\State {Disable auto-differentiation}

\vspace{0.2cm}
\State \label{algo:power:alpha-start} $\quad$
\textit{Computation of $\rho>\overline{\lambda}_{\max} \big( \sympart{\Jac_{R_{\op}}}(\widetilde{x}) \big)$:}
\State $\quad\quad$
$u_0 \leftarrow$ realization of $\Ndist(0,\Id)$
\State $\quad\quad$
    \textbf{for} $k = 1,\ldots, N_{\rm iter}$ \textbf{do}
    \vspace{0.1cm}
    \State $\quad\quad\quad$
    $u_{k+1} = \frac{  \sympart{\Jac_{R_{\op}}}(\widetilde{x}) u_k}{\|u_k\|_2^2}$ 
\State \label{algo:power:alpha-end} $\quad\quad$
\textbf{end for} 
    \State $\quad\quad$ Choose
    $\widehat{\rho} > \frac{u_{k+1}^\top \sympart{\Jac_{R_{\op}}}(\widetilde{x}) u_{k+1}}{\|u_{k+1}\|_2^2}$ \label{algo:power:jacstep:i}
\vspace{0.2cm}
\State \label{algo:power:lambda-start} $\quad$
\textit{Computation of the eigenvector associated with $\chi = \overline{\lambda}_{\max} \big( \rho \Id - \sympart{\Jac_{R_{\op}}}(\widetilde{x}) \big)$:}
\State $\quad\quad$
$v_0 \leftarrow$ realization of $\Ndist(0,\Id)$
\vspace{0.1cm}
\State $\quad\quad$
    \textbf{for} $k = 1,\ldots,N_{\rm iter}$ \textbf{do}
    \vspace{0.1cm}
    \State $\quad\quad\quad$
    $v_{k+1} = \frac{ \big( \widehat{\rho} \Id - \sympart{\Jac_{R_{\op}}}(\widetilde{x}) \big) v_k}{\|v_k\|_2^2}$ \label{algo:power:jacstep:ii}
\State $\quad\quad$ \label{algo:power:no-diff-end}
\textbf{end for} 

\vspace{0.3cm}
\State \label{algo:power:auto-diff-on}
Enable auto-differentiation

\vspace{0.2cm}
\State $\quad$
\textit{Computation of $\chi$:}
\State \label{algo:power:lambda-end} $\quad\quad$
$\widehat{\chi} = \frac{v_{N_{\rm iter}+1}^\top \big( \widehat{\rho} \Id - \sympart{\Jac_{R_{\op}}}(\widetilde{x}) \big) v_{N_{\rm iter}+1}}{\|v_{N_{\rm iter}+1}\|_2^2}$ \label{algo:power:jacstep:iii}

\vspace{0.4cm}
\State \textbf{Output:}
$\widehat{\rho} - \widehat{\chi} \simeq \lambda_{\min}\big( \sympart{\Jac}_{R_{\op}}(\widetilde{x}) \big)$ 
\end{algorithmic}
\end{algorithm}

\paragraph{Penalization computation}
Let $\widetilde{x} \in \mathbb{D}_{\rm penal}$.
We need to evaluate and subdifferentiate $\theta \in \Theta \mapsto \mathcal{P}(\theta, \widetilde{x})$, and in particular $\lambda_{\min}(\sympart{\Jac}_{R_{\op}}(\widetilde{x}))$. 
As mentioned in Remark~\ref{re:monot}\ref{re:monotiii}, \eqref{e:penalizationraw} can be reexpressed 
by using maximum absolute eigenvalues instead of a minimum eigenvalue.
Similarly to \cite{pesquet_learning_2021}, we will use a power iterative method that is designed to provide the largest eigenvalue in absolute value of a given matrix. 
The power iteration makes the computation tractable since only vector products are necessary. In particular, in our simulations, we will use the Jacobian-vector product (JVP) in Pytorch, using automatic differentiation. 
It follows from \eqref{cor:mono-sym:def-lmin} that we actually need to
combine two power iterations. The pseudo-code of the resulting method is given in Algorithm~\ref{a:computingsmallestev}. 
In particular, we first compute $\rho\ge \overline{\lambda}_{\max} \big( \sympart{\Jac_{R_{\op}}}(\widetilde{x}) \big)$ using $N_{\rm iter}$ power iterations in steps~\ref{algo:power:alpha-start}-\ref{algo:power:alpha-end} of Algorithm~\ref{a:computingsmallestev}, followed by a second run of the power iterative method to compute $\chi = \overline{\lambda}_{\max} \big( \rho \Id - \sympart{\Jac_{R_{\op}}}(\widetilde{x}) \big)$ in steps~\ref{algo:power:lambda-start}-\ref{algo:power:lambda-end} of Algorithm~\ref{a:computingsmallestev}. 
Further, in steps~\ref{algo:power:jacstep:i}, \ref{algo:power:jacstep:ii} and~\ref{algo:power:jacstep:iii}, the dot product between the Jacobian and a vector is computed using JVP as mentioned above. JVP is a function that takes three inputs: an operator $T \colon \defspace \to \defspace$ and two vectors $(\widetilde{x}, u) \in (\defspace)^2$. 
It returns $\Jac_{T}(\widetilde{x})u$. Note that in order to compute $\sympart{\Jac}_T(\widetilde{x})u$, we use the 
\href{https://j-towns.github.io/2017/06/12/A-new-trick.html}{double backward trick}\footnote{\href{https://j-towns.github.io/2017/06/12/A-new-trick.html}{https://j-towns.github.io/2017/06/12/A-new-trick.html}} and the transpose of the vector-Jacobian product.

Since all the operations used in Algorithm~\ref{a:computingsmallestev} are differentiable operations, a subgradient of $\theta \in \Theta \mapsto \mathcal{P}(\theta, \widetilde{x})$ can be estimated through auto-differentiation to enable its use in standard optimizers (see next paragraph). 
However, according to \cite{wang_backpropagation-friendly_2019}, the resulting gradient may be noisy and computationally heavy. 
To overcome these issues, we disable the track of the auto-differentiation of the forward operators (using \textit{torch.no\_grad()} function in Pytorch) to compute the needed eigenvalue and its corresponding eigenvector (steps~\ref{algo:power:alpha-start}-\ref{algo:power:no-diff-end}). Then, we activate the auto-differentiation tracking to compute the Rayleigh quotient (step~\ref{algo:power:auto-diff-on}). 
Note that this approximation is equivalent to initializing the power method with the correct eigenvector and then perform one step of the algorithm. 
In Section~\ref{Sec:inverse-train-results}, we will show that the current approximation is accurate enough to allow us to train networks such that the constraint $\lambda_{\min}(\sympart{\Jac}_{R_{\op}}(\widetilde{x}))\ge -1$ is properly satisfied.

\begin{remark}\
\begin{enumerate}
    \item 
    The use of the power iteration has the advantage of bypassing the computation of the Jacobian at a given point, which can be huge depending on the input/output dimensions. 
    \item 
    As a reminder, we cannot decompose our operator into multiple monotone operators as in \cite{runje_constrained_2023}, as the composition of two monotone operators is not necessarily monotone. Thus, being able to enforce monotonicity on a large model in an end to end fashion is required. 
    \item In order to accelerate the process, we only compute the monotonicity penalization on one random element of each batch, which experimentally yields similar results to using a batch of data, with the benefit of speeding up the training.
\end{enumerate}
\end{remark}

\smallskip

\paragraph{Generic training framework}
We present the pseudo-code summarizing the overall learning procedure in Algorithm~\ref{a:learningpseudocode}. At each epoch $j\in \{1, \ldots, N_{\rm epochs}\}$, we iteratively select $\mathbb{B}$ from $\mathbb{D}_{\rm train}$ of size $B$ until all the data has been seen. 
The associated batch loss $\ell_\mathbb{B}$ (see step \ref{a:learningpseudocode:loss} in Algorithm~\ref{a:learningpseudocode}) is associated with a Pytorch object containing all the history of the operations (called \href{https://pytorch.org/blog/computational-graphs-constructed-in-pytorch/}{\textit{computational graph}})\footnote{\href{https://pytorch.org/blog/computational-graphs-constructed-in-pytorch/}{https://pytorch.org/blog/computational-graphs-constructed-in-pytorch/}} used to compute the loss value. It is obtained by simply running a forward step with the neural network $\op$, and the associated loss and penalization values. The computational graph is then used to compute  the subgradient $g_\mathbb{B}$ through auto-differentiation (step~\eqref{a:learningpseudocode:grad} in Algorithm~\ref{a:learningpseudocode}). Finally, the optimizer step (see step~\ref{a:learningpseudocode:optim} in Algorithm~\ref{a:learningpseudocode}) updates the parameters $\theta$ of $\op$ with the selected scheme (e.g., gradient descent
with momentum, etc) \footnote{For example, the Pytorch implementation of Adam  can be found in \linebreak\href{https://pytorch.org/docs/stable/generated/torch.optim.Adam.html}{https://pytorch.org/docs/stable/generated/torch.optim.Adam.html}.}.

\begin{algorithm}
\caption{Training a monotone network $\op$}
\label{a:learningpseudocode}
\begin{algorithmic}
\State \textbf{Input:}
\begin{itemize}
    \item $\op$, $N_{\rm epochs}$, $B$, $\Delta \xi>0$  \Comment{\textit{Training parameters}}
    \item $\mathcal{L}$, $\mathcal{P}$, $\mathbb{D}_{\rm train}$ \Comment{\textit{Loss, penalization and training set}}
    \item Optimizer step: 
    $\mathcal{O} \colon (\theta,g) \mapsto \theta^+$  
    \Comment{\textit{e.g., Adam, SGD, etc.}}
\end{itemize}

\vspace{0.3cm}
\State $\xi \leftarrow 0$

\vspace{0.2cm}
\For {$j = 1,\ldots,N_{\rm epochs}$}
\vspace{0.1cm}
    \For {each batch $\mathbb{B} = \{ (x_b, y_b)\}_{1\le b \le B} \subset \mathbb{D}_{\rm train} $ of size $B$}
        
        \vspace{0.2cm}
        \State
        \textit{Computational graph related to the loss and the penalization:}
        \vspace{0.1cm}
        \State $\quad$
        $b_0 \leftarrow$ realization of discrete random uniform variable in $ \{ 1,\ldots,B \}$
        \State $\quad$
        $\nu \leftarrow$ realization of random uniform variable in $[0, 1]$
        \vspace{0.1cm}
        \State $\quad$
        $\widetilde{x}_{b_0} \leftarrow \nu x_{b_0} + (1 - \nu) y_{b_0}$
        \State $\quad$
        $\ell_{\mathbb{B}}\colon \vartheta \mapsto \frac1B \sum_{b=1}^{B}\mathcal{L}(\opsymbol_\vartheta(x_b), y_b) + \xi \, \mathcal{P}(\vartheta, \widetilde{x}_{b_0})$ 
        \Comment{\textit{Use Algorithm~\ref{a:computingsmallestev}}}
        \label{a:learningpseudocode:loss}
        
        \vspace{0.2cm}
        \State
        \textit{Gradient computation and optimizer step:}
        \vspace{0.1cm}
        \State $\quad$ \label{a:learningpseudocode:grad}
        $g_\mathbb{B}(\theta) \in  \partial \ell_\mathbb{B} (\theta)$
        \State $\quad$ \label{a:learningpseudocode:optim}
        $\theta \leftarrow \mathcal{O} (\theta, g_\mathbb{B}(\theta))$
    \EndFor
    \vspace{0.2cm}
    \State $\xi \leftarrow \xi + \Delta\xi$ \Comment{\textit{Increase penalization parameter}}
\EndFor

\vspace{0.3cm}
\State \textbf{Output:} $\op$
\end{algorithmic}
\end{algorithm}


\section{Learning non-linear model approximations}
\label{Sec:simulations}

In this section, we will study the challenging problem of learning the unknown forward model in a non-linear inverse imaging problem
and propose a method to solve this problem. 
The code associated with our experiments, including the training of monotone NNs is available on \href{https://github.com/Youyoun/truly_monotone_operator}{GitHub}.

We consider a non-linear inverse problem, where the objective is to find an estimate $\widehat{x} \in C$, with $C \subset \defspace$, of an original unknown image $\overline{x} \in C$ from degraded measurements $y\in \mathbb{R}^m$ obtained as
\begin{equation}\label{e:generalinverseproblem}
    y = \F(\overline{x}) + w,
\end{equation}
where $\F \colon \RR^n \to \RR^n$ models 
a measurement operator (possibly unknown)
, and $w \in \RR^n$ is a realization of an additive i.i.d. white Gaussian random variable with zero-mean and standard deviation $\sigma \ge 0$.

In this section, we assume that the measurement operator $\F$ is unknown. In turn, we have access to a collection of pairs of true images and associated measurements $(\overline{x},y) \in \mathbb{D}_{\rm train}$, that we will use to learn $\F$. Further, since the objective is 
to estimate $\widehat{x}$
, we will see next that it is judicious
to learn a monotone approximation of $\F$.
In the following, we will assume that $C$ is closed convex with a nonempty interior.

\subsection{Measurement operator setting}
\label{Ssec:models-F}

In the remainder, we assume that the measurement operator $\F$ is a sum of $K \in \NN$ simple non-linear composite functions, where the inner terms correspond to linear convolution filters. Precisely, we consider
\begin{equation}\label{e:generalF}
    (\forall x \in \RR^n) \quad 
    \F(x) = \frac{1}{K} \sum_{k=1}^{K}
    S_\delta(L_kx), 
\end{equation}
where $S_\delta$ is a saturation function (e.g. Hyperbolic tangent function) with parameter $\delta>0$ and, for every $k \in \{1, \ldots, K\}$, $L_k \in \RR^{n\times n}$ (e.g., $L_k$ may correspond to a convolution operator). 
Such a model has been considered for instance in~\cite{combettes2022variational}, in the particular case when $K=1$, where $L_1$ is known, and $S_\delta$ is assumed to be firmly nonexpansive.

In our experiments, we consider model~\eqref{e:generalinverseproblem}-\eqref{e:generalF} with either $K=1$ or $K=5$ motion blur kernels, of size $9 \times 9$, generated randomly using the \href{https://github.com/LeviBorodenko/motionblur}{motionblur toolbox}\footnote{\href{https://github.com/LeviBorodenko/motionblur}{https://github.com/LeviBorodenko/motionblur}}. The $5$ generated kernels are displayed in Figure~\ref{figure:Kernels}. These kernels are all positive and are normalized so that the sum of the components is equal to $1$.
In addition, we choose $S_\delta$ to be a modified $\tanh$ function 
(more details are given in Appendix~\ref{Sec:SaturationFunction}), with $\delta \in \{1, 0.6\}$. 
Finally, we will consider two noise levels with standard deviation $\sigma \in \{0, 10^{-2}\}$.

Note that there is no guarantee that model~\eqref{e:generalF} is monotone. 
Solving inverse problems with such an operator can appear to be challenging, especially due to the non-linearity $S$. 
A standard practice in such an inverse problem context consists in replacing $F$ by its first-order approximation expressed as 
\begin{equation}\label{e:generalFaff}
    (\forall x \in \RR^n) \quad \Faff(x) = \frac{\delta}{K} \sum_{k=1}^{K}L_kx + \dfrac{1-\delta}{2}. 
\end{equation}
For simplicity, 
when the parameter $\delta$ is assumed to be unknown, we can use instead the following linear approximation:
\begin{equation}\label{e:generalFlinear}
    (\forall x \in \RR^n) \quad \Flin(x) = \frac{\delta}{K} \sum_{k=1}^{K}L_kx. 
\end{equation}
Such a linearization of the model is standard in inverse problems. Note that it does not introduce any bias when $\delta = 1$.

Since the convolution filters are in practice unknown, $\Flin$ represents the linear oracle. For this purpose, we will be approximating $F$ using a single linear operator $\linop$, in order to compare our method to a more realistic result. More details about the results of this approximation are provided in Appendix~\ref{sec:linear-approx}.

An example is provided in Figure~\ref{fig:degrad_fn_example}, for the case when $K=5$, $\sigma=0$, and $S_\delta$ is chosen as in Appendix~\ref{Sec:SaturationFunction}. From left to right, the figure shows the true image $\overline{x}$, the observed image $y = F(\overline{x})$, and the approximated linear degradation obtained as $\widetilde{y} = \Flin(\overline{x})$. It can be observed that the light background behind the penguin is darker when using the true acquisition model $F$ than with the linear version $\Flin$, due to the saturation function $S_\delta$.

\begin{figure}[!h]
    \centering
    \setlength{\tabcolsep}{0.1cm}
    \footnotesize
    \begin{tabular}{ccccc}
    \includegraphics[width=0.15\linewidth]{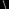}
    &   \includegraphics[width=0.15\linewidth]{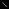}
    &   \includegraphics[width=0.15\linewidth]{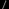}
    &   \includegraphics[width=0.15\linewidth]{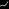}
    &   \includegraphics[width=0.15\linewidth]{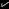} \\
    $L_1$ & $L_2$ & $L_3$ & $L_4$ & $L_5$ 
    \end{tabular}

    \vspace*{-0.2cm}
    \caption{\small Blurring kernels used to model linear operators $(L_k)_{1\le k \le 5}$, used in model~\eqref{e:generalF}.}%
    \label{figure:Kernels}
\end{figure}
\begin{figure}[!h]
    \centering
    \setlength{\tabcolsep}{0cm}
    \footnotesize
    \begin{tabular}{cccc}
        \includegraphics[height=2.1cm]{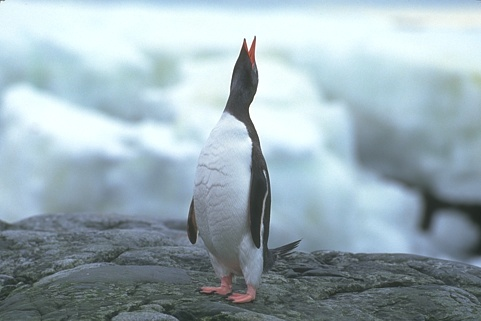}
        &   
        \includegraphics[height=2.1cm]{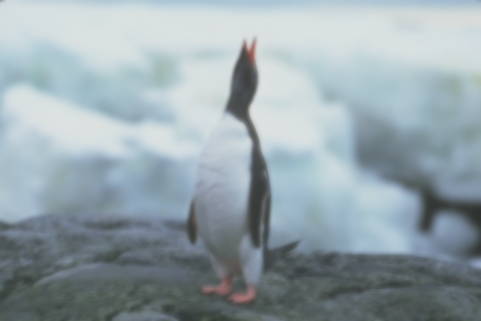}
        &   
        \includegraphics[height=2.1cm]{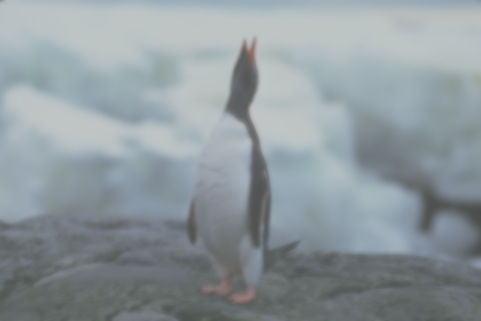} 
        &
        \includegraphics[height=2.1cm]{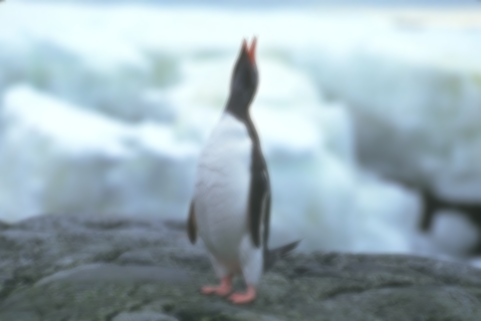} \\
        $\overline{x}$ & $y=F(\overline{x})$ & $y=F(\overline{x})$ & $\widetilde{y} = \Flin(\overline{x})$  \\ 
         & $\delta=1$ & $\delta=0.6$ & 
    \end{tabular}

    \vspace*{-0.2cm}
    \caption{\small Example of an original image $\overline{x}$, the observation of this image through \eqref{e:generalF} with $K=5$ and $\sigma=0$, and the linearized observation of $\overline{x}$ through~\eqref{e:generalFlinear}. }
    \label{fig:degrad_fn_example}
\end{figure}

\subsection{Monotone inclusion formulation to solve~\eqref{e:generalinverseproblem}}
\label{sssec:sim:image:mono}

The objective of this section is to find an estimate $\widehat{x}$ of $\overline{x}$ from $y$, solving the inverse problem~\eqref{e:generalinverseproblem}, i.e., inverting $\F$. Note that the approach described below is not restricted to the particular choice of operator $F$ described in Section~\ref{Ssec:models-F}.

The problem of interest is equivalent to
\begin{equation}
\mbox{find}\; \widehat{x} \in C 
\text{ such that } \F(\widehat{x}) \approx y.
\end{equation}
We propose two approaches to solve this problem.

\smallskip
\paragraph{\textbf{Direct regularized approach}}
We propose to define $\widehat{x}$ as a solution to a regularized monotone inclusion problem of the form:
\begin{equation}\label{e:maininverse}
    \text{find } \widehat{x} \in \defspace \text{ such that } 0 \in \op(\widehat{x}) - y + \varrho \nabla r(\widehat{x}) + N_C(\widehat{x}),
\end{equation}
where $\op \colon \defspace \to \defspace$ is a continuous monotone approximation to $\F$ on $C$, $\varrho\in \RP$, and 
$r \colon \RR^n \to \RR$ is a differentiable and convex regularization function. 

Interestingly, \eqref{e:maininverse} is an instance of Problem~\ref{prob:main}, where
\begin{equation}\label{e:problem_functions}
    (\forall x \in \RR^n)\quad
    \begin{cases}
        h(x) = -\scal{y}{x}\\ 
        A(x) = \op(x) + \varrho \nabla r(x).
    \end{cases}
\end{equation}
We have the following result concerning the existence of
a solution to \eqref{e:maininverse}.
\begin{proposition}\label{prop:existsolinvprob}
Assume that one of the following statements holds:
\begin{enumerate}
    \item $\rho > 0$ and $r$ is strongly convex,
    \item $C$ is bounded.
\end{enumerate}
Then there exists a solution $\widehat{x}$ to \eqref{e:maininverse}.
\end{proposition}
\begin{proof}
Since $A$ is monotone on $C$, it admits a maximally monotone extension $\widetilde{A}$ on $\RR^n$. Problem
\eqref{e:maininverse} is thus equivalent to find
a zero of $\widetilde{A}+\nabla h+N_C$. 
By proceeding similarly to the proof of Proposition~\ref{p:conv} 
we can show that this operator is maximally monotone.
In addition, if $\rho> 0$ and $r$ is strongly convex, it is strongly monotone. The existence of a solution
to \eqref{e:maininverse} follows then from
\cite[Proposition~23.36(ii)]{bauschke_convex_2017} 
and \cite[Corollary~23.37(ii)]{bauschke_convex_2017}
\end{proof}

To solve \eqref{e:maininverse} numerically, we can make use of
Algorithm~\ref{a:FBF-pnp}.
The convergence of this algorithm is guaranteed by Proposition~\ref{p:conv}, as soon as condition \ref{p:convii} on the choice of the step-size is satisfied.

\begin{remark}\label{rmk:invert_nn}
In the particular case when $y$ is replaced by $\op(\overline{x})$ in~\eqref{e:maininverse}, Algorithm~\eqref{a:FBF} enables computing  a regularized inverse of $\op$ delivering an estimate
$x^*$ to $\overline{x}$.
In particular, if $\rho=0$ and $x^*\in \inte C$, $x^* = \op^{-1}(\op(\overline{x}))$.

\end{remark}

\smallskip
\paragraph{\textbf{Least-squares regularized approach}}
In addition to the direct approach described above, we will investigate an original reformulation of problem~\eqref{e:maininverse} as
\begin{equation}
    \text{Find } \widetilde{x} \in \defspace 
    \text{ such that } 
    0 \in {\Flin}^{\top} \op(\widetilde{x}) - \widetilde{y} + \varrho \nabla r(\widetilde{x}) + N_C(\widetilde{x}) \label{e:maininverse-leastsquares} 
\end{equation}
where $\widetilde{y} = {\Flin}^{\top} y$ and $\op \colon \defspace \to \defspace$ is a continuous approximation to $\F$ on $C$. 
This approach is potentially less restrictive than \eqref{e:maininverse} as it does not impose the approximation to the linear model to be monotone, but
requires 
only ${\Flin}^{\top} \op$ to be monotone.
Interestingly, when $\op = \Flin$,
\eqref{e:maininverse-leastsquares} corresponds to the standard least-squares formulation. 
Then the existence of a solution to \eqref{e:maininverse-leastsquares} is guaranteed similarly to Proposition~\ref{prop:existsolinvprob}.
Algorithm~\ref{a:FBF} can be rewritten for solving \eqref{e:maininverse-leastsquares}, similarly to \eqref{e:TsengalgoGeneral}.

\begin{remark}\label{rmk:leastsquares}\
\begin{enumerate}
\item 
Similarly to Remark~\ref{rmk:invert_nn}, operator $A$ can be inverted using~\eqref{e:maininverse-leastsquares} by setting $\widetilde{y} = {\Flin}^{\top} \overline{x}$ and $\varrho = 0$. 
\item\label{rmk:leastsquares:ii}
According to \cite[Prop.~20.10]{bauschke_convex_2017},
when $K=1$, ${\Flin}^{\top} F=L_1^\top\circ S_\delta \circ L_1$ is monotone.

\end{enumerate}
\end{remark}

\smallskip
\paragraph{\textbf{Learned approximations to $F$}}
In this work, we assume that neither the kernels $(L_k)_{1 \le k \le K}$ nor the parameter $\delta$ are known in~\eqref{e:generalF}.
Then, we leverage both the direct and the least-squares formulations described above to derive our learned approximation strategies for $F$. The different considered approaches are described below:
\begin{enumerate}
\item 
The first classical strategy consists in learning a linear approximation $\linop$ of $F$. 
\item 
We leverage the training approach developed in Section~\ref{ssec:proposedapproach:reg} to learn a monotone approximation $\monop$ of $F$.
\item 
To show the necessity of the proposed constrained training approach to obtain a monotone operator, we also consider an approximation $\nmonop$ of $F$, that is learned without the proposed regularization (i.e., $\xi=0$ in \eqref{e:min:loss-pen-gen}).
\item 
We finally leverage the least-squares formulation~\eqref{e:maininverse-leastsquares} combined with the training approach developed in Section~\ref{ssec:proposedapproach:reg} to learn a monotone operator ${\linop}^{\top} \op$. Since $\Flin$ is unknown, we use the learned linear operator $\linop$, and learn $\op$ under the constraint that ${\linop}^{\top} \op$ is monotone. In the remainder, we will use the notation $\monopsym = {\linop}^{\top}\op$ when referring to this approximation strategy.
\item Similarly, the least squares approach can be used with $\linop$, defining $\linopsym = {\linop}^{\top} \linop$ which is monotone.
\end{enumerate}

\begin{remark}\
\begin{enumerate}
\item In the last restoration approach based on 
the least-squares strategy \eqref{e:maininverse-leastsquares},
$\Flin$ is replaced by $\linop$ and $\widetilde{y} = {\linop}^{\top} y$.
\item 
We emphasize that none of the proposed restoration procedures can be used with $\nmonop$ without losing theoretical guarantees, since neither $\nmonop$ nor $({\linop}^\top \nmonop)$ are monotone. 
\end{enumerate}
\end{remark}

We summarize all the different forward models and notation used in this section in Table~\ref{tab:model-summary}.

\begin{table}[h!]
    \centering\small
    \setlength\tabcolsep{0.1cm}
    \setlength\extrarowheight{0.1cm}
    \begin{tabular}{lp{6.6cm}|cc}
        \hline
        Model 
        & Description & Monotone & Learned \\
        \hline
        $\F$ 
        & True model (see \eqref{e:generalF}) & No & No \\
        $\Faff$ 
        & Affine approximation (see \eqref{e:generalFaff}) & No$^\ast$ & No \\
        $\Flin$ 
        & Linear approximation (see \eqref{e:generalFlinear}) & No$^\ast$ & No \\
        $\linop$
        & Learned linear approximation (see Section~\ref{sssec:sim:image:mono}, \textit{Learned approximations to $F$ (i)}) & No* & Yes \\
        $\linopsym = {\linop}^{\top} \linop$
        & Learned linear approximation using least-squares formulation~\eqref{e:maininverse-leastsquares} (see Section~\ref{sssec:sim:image:mono}, \textit{Learned approximations to $F$ (v)})
        & Yes & Yes \\ 
        $\monop$
        & Learned monotone approximation (see Section~\ref{sssec:sim:image:mono}, \textit{Learned approximations to $F$ (ii)}) & Yes & Yes \\
        $\nmonop$
        & Learned approximation without monotone constraint (see Section~\ref{sssec:sim:image:mono}, \textit{Learned approximations to $F$ (iii)}) & No & Yes \\
        $\monopsym = {\linop}^{\top}\op$
        & Learned approximation using least-squares formulation~\eqref{e:maininverse-leastsquares} (see Section~\ref{sssec:sim:image:mono}, \textit{Learned approximations to $F$ (iv)}) & Yes & Yes \\
        \hline
    \end{tabular}
    \caption{Summary of the different models and notation used for our simulations. $^\ast$Note that approximations $\Faff$ and $\Flin$ are generally not monotone.
   Motonicity is garanteed if $(L_k)_{1 \le k \le K}$ are positive semi-definite operators. Similarly $\linop$ is not necessarily monotone if no regularization is imposed during training.}
    \label{tab:model-summary}
\end{table}

\subsection{Model and training procedure}

In this section we describe the learning procedures adopted for the considered approximations to $F$ described in the previous section, focusing on non-linear approximations. %
The learning procedure of the linear operator $\linop$ is given in Section~\ref{sec:linear-approx}.

\smallskip

\paragraph{NN architecture} 
We use a Residual Unet \cite{zhang_road_2018} consisting of $5$ blocks where the output of each down convolution have $32$, $64$, $128$, $256$, and $512$ feature maps, respectively. LeakyReLU activations were used for the intermediate layers. No activation was used in the last layer, so the learned NN must model the saturation using its weights and biases.

\smallskip

\paragraph{Training dataset}
For $\mathbb{D}_{\rm train}$, we use patches of size $256 \times 256$, randomly extracted from the BSD500 dataset, with pixel-values scaled to the range $[0, 1]$. 
The couples $(\overline{x},y)$ in $\mathbb{D}_{\rm train}$ are linked through model~\eqref{e:generalinverseproblem}-\eqref{e:generalF}. To investigate the stability against noise, we consider two cases for the training set: (i) training without noise (i.e., $\sigma_{\rm train} = 0$), and (ii) training with noise level $\sigma_{\rm train} =0.01$. \\
For $\mathbb{D}_{\rm penal}$, we crop the images of $\mathbb{D}_{\rm train}$ to smaller patches of size $64 \times 64$.
We set $68$ images from the dataset as test images (corresponding to BSD68), and use the other images for training.

\smallskip

\paragraph{Training procedure}
The NN is trained using Adam optimizer with a learning rate of $2\times 10^{-4}$ for $200$ epochs to solve \eqref{e:min:loss-pen-gen} with an $\ell_1$ loss function:
\begin{equation}\label{e:l1loss}
   \mathcal{L}(\op(x), y) = \|\op(x) - y\|_1.
\end{equation}
We use a learning rate scheduler that decreases the learning rate by a factor of $0.1$ when the training loss reaches a plateau. Model $\nmonop$ is trained by setting $\xi=0$ in \eqref{e:min:loss-pen-gen} and $\Delta \xi = 0$, and Model $\monop$ is obtained by setting $\xi=0.1$ in \eqref{e:min:loss-pen-gen} at start and $\Delta \xi = 0.1$. Finally, $\epsilon$ is chosen equal to $0.01$ in~\eqref{e:penalizationraw}.

\begin{table}[!h]
    \centering\small
    \begin{tabular}{lllcc}
        \toprule
        Filters   &  Noise & Model &  $\text{MAE}(y, \op(\overline{x}))$ &  $\min \lambda_{\text{min}} \big( \sympart{\Jac_{\op}}(\overline{x}) \big)$ \\
        &&& ($\times 10^{-2}$) & ($\times 10^{-2}$) \\
        \midrule

        \midrule
        \multicolumn{5}{c}{$\delta = 1$ for $S_\delta$ in~\eqref{e:generalF}} \\ 
        \midrule
        \multirow{8}{*}{$K=1$}
        & \multirow{4}{*}{$\sigma_{\rm train}=0$} &  $\monop$ & 1.5658 ($\pm$ 0.58) & 1.67 \\
         & & $\nmonop$  & 0.2932 ($\pm$ 0.11) & -29.99 \\
         & & $\monopsym$  & 0.6822 ($\pm$ 0.25) & $0.69^*$ \\
         & & $\linop$  & 2.1474 ($\pm$ 1.38) & -24.69 \\
         \cline{2-5}
         & \multirow{4}{*}{$\sigma_{\rm train}=0.01$} & $\monop$ & 1.1575 ($\pm$ 0.42) & 1.10   \\
         & & $\nmonop$ & 0.3020 ($\pm$ 0.11) & -27.72 \\
         & & $\monopsym$  & 0.5351 ($\pm$ 0.19) & $1.13^*$ \\
         & & $\linop$  & 2.1607 ($\pm$ 1.37) & -25.87 \\
        \hline
        \multirow{8}{*}{$K=5$}
        &  \multirow{4}{*}{$\sigma_{\rm train}=0$} & $\monop$ & 0.5272 ($\pm$ 0.20) & 1.18 \\
         & & $\nmonop$ & 0.2795 ($\pm$ 0.10) & -22.06 \\
         & & $\monopsym$ & 0.6108 ($\pm$ 0.22) & $1.80^*$ \\
         & & $\linop$  & 2.1473 ($\pm$ 1.38) & -13.86 \\
         \cline{2-5}
         & \multirow{4}{*}{$\sigma_{\rm train}=0.01$} & $\monop$ & 0.9414 ($\pm$ 0.34) & 1.80 \\
         && $\nmonop$ & 0.2714 ($\pm$ 0.09) & -25.48 \\
         & & $\monopsym$  & 0.6591 ($\pm$ 0.25) & $1.64^*$ \\
         & & $\linop$  & 2.1594 ($\pm$ 1.37) & -16.55 \\

        \midrule
        \multicolumn{5}{c}{$\delta = 0.6$ for $S_\delta$ in~\eqref{e:generalF}} \\ 
        \midrule
        \multirow{8}{*}{$K=1$}
        & \multirow{4}{*}{$\sigma_{\rm train}=0$} &  $\monop$ & 1.2816 ($\pm$ 0.53) & 1.91 \\
         & & $\nmonop$  & 0.2376 ($\pm$ 0.09) & -18.73 \\
         & & $\monopsym$  & 0.4311 ($\pm$ 0.14) & $0.37^*$ \\
         & & $\linop$  & 8.2009 ($\pm$ 2.70) & -42.79 \\
         \cline{2-5}
         & \multirow{4}{*}{$\sigma_{\rm train}=0.01$} & $\monop$ & 1.1689 ($\pm$ 0.45) & 1.65 \\
         & & $\nmonop$ & 0.2275 ($\pm$ 0.08) & -23.35 \\
         & & $\monopsym$  & 0.4679 ($\pm$ 0.17) & $0.54^*$ \\
         & & $\linop$  & 8.1993 ($\pm$ 2.69) & -43.18 \\
        \hline
        \multirow{8}{*}{$K=5$}
        &  \multirow{4}{*}{$\sigma_{\rm train}=0$} & $\monop$ & 0.7720 ($\pm$ 0.30) & 1.52 \\
         & & $\nmonop$ & 0.1435 ($\pm$ 0.05) & -17.38 \\
         & & $\monopsym$  & 0.4327 ($\pm$ 0.14) & $0.40^*$ \\
         & & $\linop$  & 8.1920 ($\pm$ 2.70) & -39.61 \\
         \cline{2-5}
         & \multirow{4}{*}{$\sigma_{\rm train}=0.01$} & $\monop$ & 0.6867 ($\pm$ 0.25) & 0.66 \\
         & & $\nmonop$ & 0.1788 ($\pm$ 0.06) & -19.04 \\
         & & $\monopsym$  & 0.4809 ($\pm$ 0.15) & $0.33^*$ \\
         & & $\linop$  & 8.1969 ($\pm$ 2.70) & -35.39 \\
        \bottomrule
    \end{tabular}
    \caption{\small Training results obtained with monotone regularization ($\monop$), and without the regularization ($\nmonop$). Results linked to the least squares operator $\monopsym$ and to the learned linear kernel $\linop$ are also presented. Metrics were computed on cropped images of size $256 \times 256$ from the BSD68 test set. \linebreak ${}^*$For $\monopsym$, the smallest eigenvalue was computed for the whole operator ${\linop}^{\top} \op$.}
    \label{tab:results_trainings}
\end{table}

\begin{figure}[!b]
    \centering
    \scriptsize
    \setlength{\tabcolsep}{0.05cm}
    \begin{tabular}{ccc}
        \includegraphics[height=2.7cm]{Figures/Simulation/Results20231127/images_F/1/F_x.png}
        & \includegraphics[height=2.7cm]{Figures/Simulation/Results20231127/images_F/1/F_lin_x.png}
        & \includegraphics[height=2.7cm]{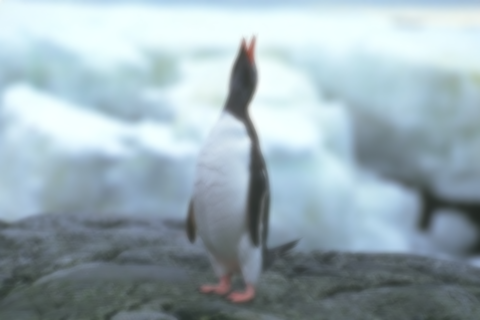} \\
        $y=F(\overline{x})$ & $y_{\Flin} = \Flin(\overline{x})$ & $y_{\linop}=\linop(\overline{x})$ \\
        PSNR $= 24.07$ & MAE($y, y_{\Flin}$) $= 0.035$ & MAE($y, y_{\linop}$) $=0.035$ \\
        $ (\lambda_{\min}, \lambda_{\max}) = (-80.52, 81.37)$ & $ (\lambda_{\min}, \lambda_{\max}) = (-0.09, 1.00)$ & $ (\lambda_{\min}, \lambda_{\max}) = (-0.14, 0.99)$ \\[0.2cm]
        \includegraphics[height=2.7cm]{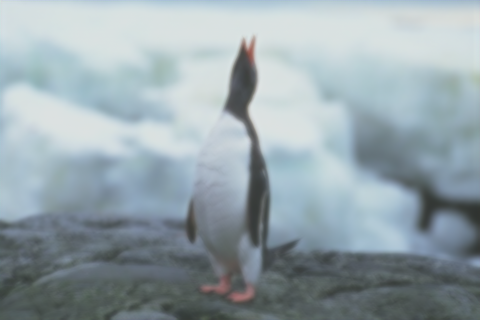} 
        & \includegraphics[height=2.7cm]{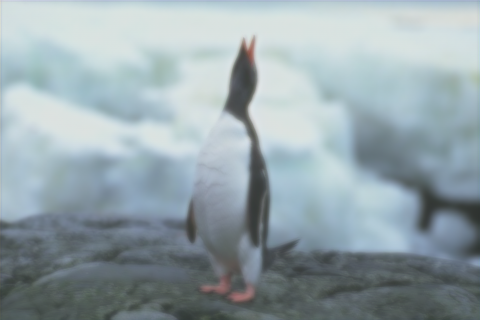}
        & \includegraphics[height=2.7cm]{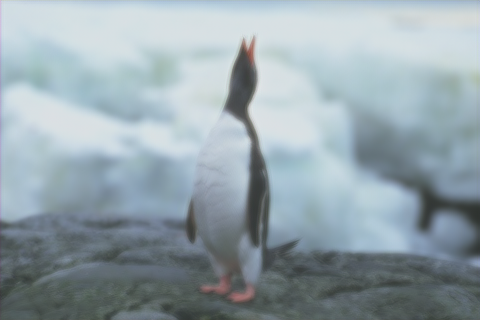} \\
        $y_{\nmonop} = \nmonop(\overline{x})$ & $y_{\monop} = \monop(\overline{x})$ & $y_{\monopsym} = \op(\overline{x})$  \\
        MAE($y, y_{\nmonop}$) $= 0.002$ & MAE($y, y_{\monop}$) $= 0.004$ & MAE($y, y_{\monopsym}$) $=0.004$ \\
        $ (\lambda_{\min}, \lambda_{\max}) = (-0.21, 1.01)$ & $(\lambda_{\min}, \lambda_{\max}) = (0.01, 1.00)$ & $(\lambda_{\min}, \lambda_{\max}) = (0.00, 0.87)$ \\[0.3cm]
        \includegraphics[height=2.7cm]{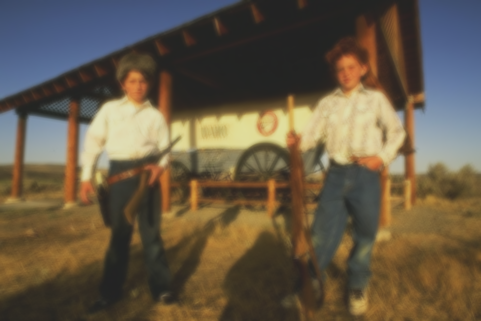}
        & \includegraphics[height=2.7cm]{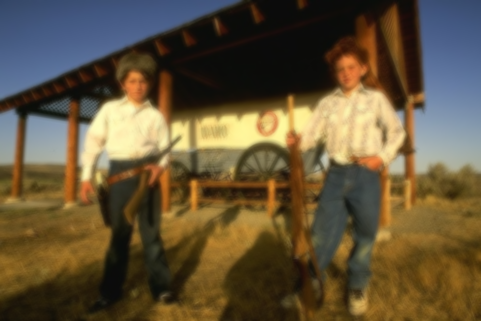}
        & \includegraphics[height=2.7cm]{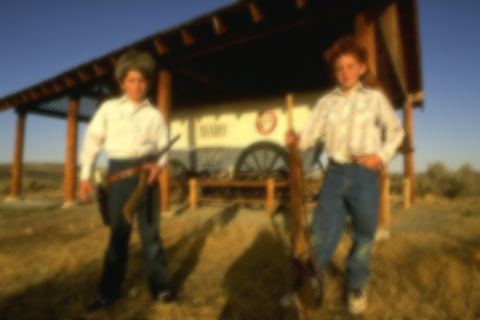} \\
        $y=F(\overline{x})$ & $y_{\Flin} = \Flin(\overline{x})$ & $y_{\linop}=\linop(\overline{x})$ \\
        PSNR $= 21.17$ & MAE($y, y_{\Flin}$) $= 0.037$ & MAE($y, y_{\linop}$) $=0.037$ \\
        $ (\lambda_{\min}, \lambda_{\max}) = (-80.52, 81.37)$ & $ (\lambda_{\min}, \lambda_{\max}) = (-0.09, 1.00)$ & $ (\lambda_{\min}, \lambda_{\max}) = (-0.14, 0.99)$ \\[0.2cm]
        \includegraphics[height=2.7cm]{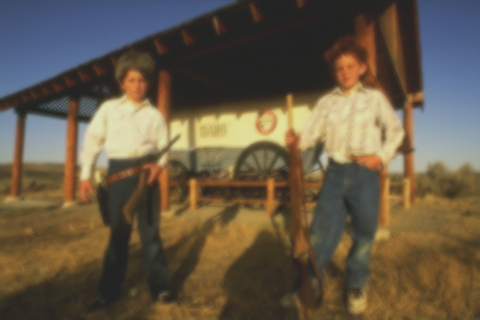} 
        & \includegraphics[height=2.7cm]{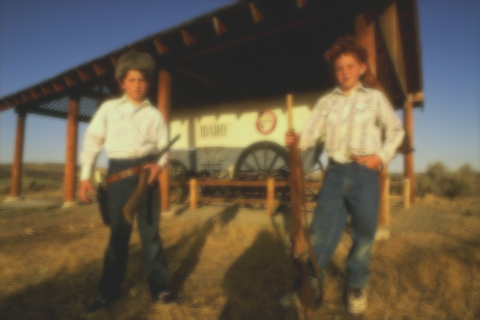}
        & \includegraphics[height=2.7cm]{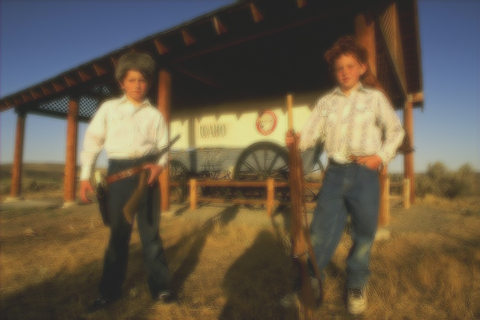} \\
        $y_{\nmonop} = \nmonop(\overline{x})$ & $y_{\monop} = \monop(\overline{x})$ & $y_{\monopsym} = \op(\overline{x})$  \\
        MAE($y, y_{\nmonop}$) $= 0.003$ & MAE($y, y_{\monop}$) $= 0.005$ & MAE($y, y_{\monopsym}$) $=0.006$ \\
        $ (\lambda_{\min}, \lambda_{\max}) = (-0.21, 1.01)$ & $(\lambda_{\min}, \lambda_{\max}) = (0.01, 1.00)$ &  $ (\lambda_{\min}, \lambda_{\max}) = (0.00, 0.92)$
    \end{tabular}
    
    \caption{\small
    Examples of output images obtained with the different versions of the measurement operator, with $\delta = 1$. First and third rows, left to right: true unknown operator $F$, true unknown linear approximation $\Flin$, learned linear approximation $\linop$. Second and fourth rows, left to right: learned non-monotone approximation $\nmonop$, proposed learned monotone approximation  $\monop$, and proposed relaxed monotone approximation $\monopsym$. 
    All results are shown when training models without noise (i.e., $\sigma_{\rm train}=0$).
    }%
    \label{f:operatorexamples-1}
\end{figure}

\begin{figure}[!b]
    \centering
    \scriptsize
    \setlength{\tabcolsep}{0.05cm}
    \begin{tabular}{ccc}
        \includegraphics[height=2.7cm]{Figures/Simulation/Results20231127/images_F/1/F_hard_x.png}
        & \includegraphics[height=2.7cm]{Figures/Simulation/Results20231127/images_F/1/F_lin_x.png}
        & \includegraphics[height=2.7cm]{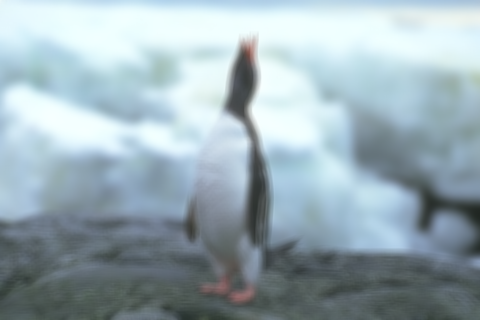} \\ 
        $y=F(\overline{x})$ & $y_{\Flin} = \Flin(\overline{x})$ & $y_{\linop}=\linop(\overline{x})$ \\
        PSNR $= 17.33$ & MAE($y, y_{\Flin}$) $= 0.110$ & MAE($y, y_{\linop}$) $=0.108$ \\
        $ (\lambda_{\min}, \lambda_{\max}) = (-156.86, 157.59)$ & $ (\lambda_{\min}, \lambda_{\max}) = (-0.09, 1.00)$ & $ (\lambda_{\min}, \lambda_{\max}) = (-0.35, 0.99)$ \\[0.2cm]
        \includegraphics[height=2.7cm]{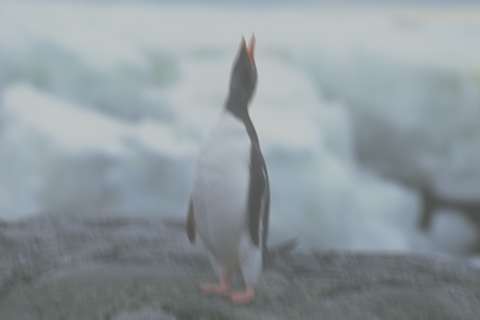} 
        & \includegraphics[height=2.7cm]{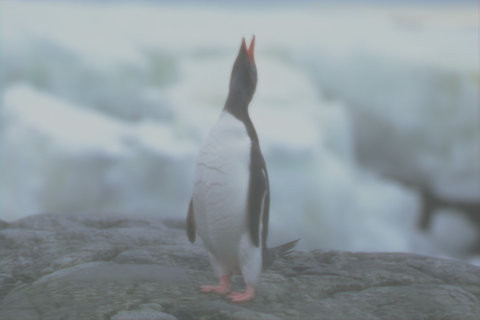}
        & \includegraphics[height=2.7cm]{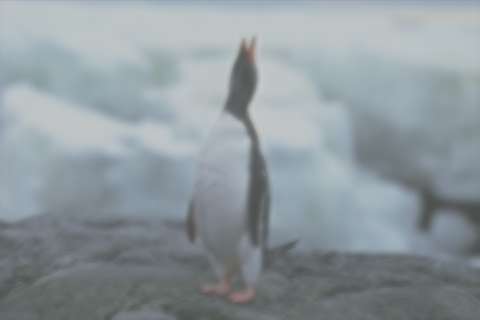} \\
        $y_{\nmonop} = \nmonop(\overline{x})$ & $y_{\monop} = \monop(\overline{x})$ & $y_{\monopsym} = \op(\overline{x})$  \\
        MAE($y, y_{\nmonop}$) $= 0.004$ & MAE($y, y_{\monop}$) $= 0.006$ & MAE($y, y_{\monopsym}$) $=0.003$ \\
        $ (\lambda_{\min}, \lambda_{\max}) = (-0.18, 0.63)$ & $(\lambda_{\min}, \lambda_{\max}) = (0.03, 0.61)$ & $(\lambda_{\min}, \lambda_{\max}) = (0.00, 0.55)$ \\[0.3cm]
        \includegraphics[height=2.7cm]{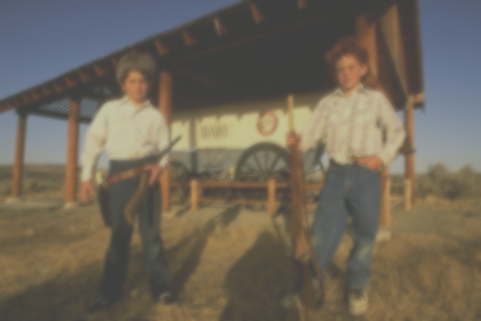}
        & \includegraphics[height=2.7cm]{Figures/Simulation/Results20231127/images_F/24/F_lin_x.png}
        & \includegraphics[height=2.7cm]{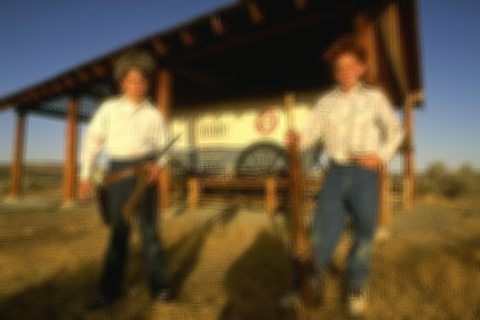} \\ 
        $y=F(\overline{x})$ & $y_{\Flin} = \Flin(\overline{x})$ & $y_{\linop}=\linop(\overline{x})$ \\
        PSNR $= 17.33$ & MAE($y, y_{\Flin}$) $= 0.110$ & MAE($y, y_{\linop}$) $=0.111$ \\ 
        $ (\lambda_{\min}, \lambda_{\max}) = (-156.86, 157.59)$ & $ (\lambda_{\min}, \lambda_{\max}) = (-0.09, 1.00)$ & $ (\lambda_{\min}, \lambda_{\max}) = (-0.35, 1.00)$ \\[0.2cm]
        \includegraphics[height=2.7cm]{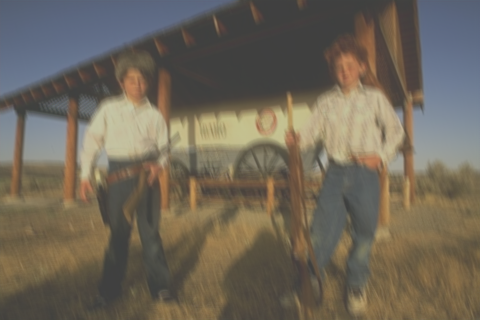} 
        & \includegraphics[height=2.7cm]{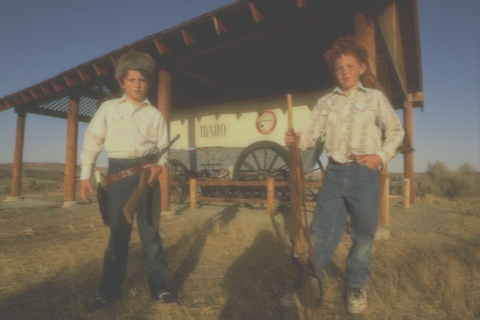}
        & \includegraphics[height=2.7cm]{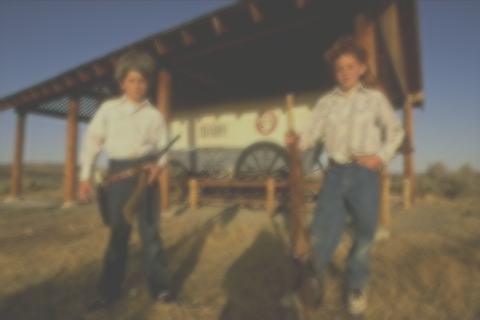} \\
        $y_{\nmonop} = \nmonop(\overline{x})$ & $y_{\monop} = \monop(\overline{x})$ & $y_{\monopsym} = \op(\overline{x})$  \\
        MAE($y, y_{\nmonop}$) $= 0.009$ & MAE($y, y_{\monop}$) $= 0.009$ & MAE($y, y_{\monopsym}$) $=0.004$ \\ 
        $ (\lambda_{\min}, \lambda_{\max}) = (-0.18, 0.63)$ & $(\lambda_{\min}, \lambda_{\max}) = (0.03, 0.61)$ &  $ (\lambda_{\min}, \lambda_{\max}) = (0.00, 0.56)$
    \end{tabular}
    
    \caption{\small
    Examples of output images obtained with the different versions of the measurement operator, with $\delta = 0.6$. First and third rows, left to right: true unknown operator $F$, true unknown linear approximation $\Flin$, learned linear approximation $\linop$. Second and fourth rows, left to right: learned non-monotone approximation $\nmonop$, proposed learned monotone approximation  $\monop$, and proposed relaxed monotone approximation $\monopsym$.
    All results are shown when training models without noise (i.e., $\sigma_{\rm train}=0$).
    }%
    \label{f:operatorexamples-0.6}
\end{figure}

\subsection{Training results}
\label{Sec:inverse-train-results}

In this section, we assess the abilities of the learned NNs in approximating the measurement operator $F$, as well as satisfying the monotonicity constraint necessary to the reconstruction problem. 
The results are provided in~Table~\ref{tab:results_trainings}. For each model, the average MAE values and the smallest eigenvalue $\min_{\overline{x}} \lambda_{\min} \big( \sympart{\Jac_{\op}}(\overline{x}) \big)$ are reported in the table, where $\overline{x}$ are cropped images of size $256 \times 256$ from the test set BSD68. 
Results are provided for the two considered training noise levels $\sigma_{\rm train} \in \{0, 0.01\}$. 
For $\monopsym$, the smallest eigenvalue was computed for the whole operator $\monopsym = {\linop}^\top \op$.
To ensure the monotonicity of the model, the smallest eigenvalue must be non-negative. 
The MAE values are computed between the non-noisy true output $y=F(\overline{x})$, and the output obtained through the learned network $\op(\overline{x})$.

On the one hand, we observe that $\nmonop$ performs slightly better in terms of MAE value than $\monop$, in particular when the $K=1$ filter is used. This behaviour is expected since the monotonicity constraint reduces the flexibility
of the model. We further observe that the noise level added during the training does not appear to have much impact on the results.
On the other hand, the learned model corresponding  to $\monopsym$ seems to better reproduce the true model $\F$, with MAE values closer to the non-monotone model. This is certainly due to the weaker constraint on its structure. 
Finally, $\linop$ is the worst approximation of $F$, and is not monotone.

Figure~\ref{f:operatorexamples-1} and Figure~\ref{f:operatorexamples-0.6} show the output images obtained when using different versions of the measurement operator, for two image examples, for $K=5$ with $\delta=1$ and $\delta=0.6$, respectively. Output images are provided for the true unknown operator $\F$, as well as for approximated operators $\Flin$ (true unknown linear approximation), $\linop$, $\nmonop$, $\monop$, and $\monopsym$. Results are shown for models trained without noise (i.e., $\sigma_{\rm train}=0$).
For each image, we give the MAE between the output from the approximated operator and the output from the true operator. Furthermore, values $(\lambda_{\min}, \lambda_{\max})$ are reported. We can observe that $\lambda_{\min} < 0$ for all operators apart from the monotone network $\monop$ and the relaxed version $\monopsym$. The true model $\F$, the linear approximation $\Flin$, and the learned linear approximation $\linop$ are all non-monotone. This shows that our method enables approximating a non-monotone operator using a monotone network.

\begin{table}[!t]
    \centering
    \setlength{\tabcolsep}{0.15cm}
    \small
    \begin{tabular}{l c l|r r | r r}
    \midrule
        & \multirow{2}{*}{Problem} & \multirow{2}{*}{Operator} & \multicolumn{2}{c|}{$\sigma_{\rm train}=0$} & \multicolumn{2}{c}{$\sigma_{\rm train}=0.01$} \\[0.1cm]
        &&& PSNR & SSIM & PSNR & SSIM  \\
    \midrule
    \parbox[t]{2mm}{\multirow{3}{*}{\rotatebox[origin=c]{90}{\scriptsize$(K, \delta)=(1, 1)$}}}
        & {\eqref{e:maininverse}} & {$\monop$} 
            &  $24.56 (\pm 3.96)$ &  $ 0.80 (\pm 0.11)$   
            &  $24.58 (\pm 4.26)$ &  $0.80 (\pm 0.11)$ \\[0.2cm]
        & {\eqref{e:maininverse-leastsquares}}& {$\monopsym$} 
            &  $26.32 (\pm 4.14)$ &  $0.85 (\pm 0.04)$ 
            &  $\best{28.31 (\pm 3.66)}$ &  $\best{0.89 (\pm 0.04)}$ \\[0.2cm]
        & {\eqref{e:maininverse-leastsquares}}  & {$\linopsym$} 
            &  $25.59 (\pm 3.14)$  &  $0.87 (\pm 0.07)$  
            & $25.59 (\pm 3.11)$  &  $0.87 (\pm 0.07)$ \\[0.2cm] 
    
    \midrule
    \parbox[t]{2mm}{\multirow{3}{*}{\rotatebox[origin=c]{90}{\scriptsize$(K, \delta)=(5, 1)$}}}
        & {\eqref{e:maininverse}}& {$\monop$} 
            & $27.46 (\pm 4.31)$ & $0.87 (\pm 0.08)$ 
            & $26.96 (\pm 4.13)$ & $0.86 (\pm 0.08)$ \\[0.2cm]
        & {\eqref{e:maininverse-leastsquares}}& {$\monopsym$} 
            &  $28.31 (\pm 4.32)$ &  $0.89 (\pm 0.06)$ 
            &  $\best{28.33 (\pm 4.33)}$ &  $\best{0.89 (\pm 0.06)}$ \\[0.2cm]
        & {\eqref{e:maininverse-leastsquares}}  & {$\linopsym$} 
            & $25.21 (\pm 3.29)$  &  $0.86 (\pm 0.08)$ 
            & $25.23 (\pm 3.32)$ & $0.86 (\pm 0.08)$ \\[0.2cm] 
    
    \midrule
    \parbox[t]{2mm}{\multirow{3}{*}{\rotatebox[origin=c]{90}{\scriptsize$(K, \delta)=(1, 0.6)$}}}
        &   {\eqref{e:maininverse}} & {$\monop$} 
            & $25.17 (\pm 3.99)$ & $0.81 (\pm 0.10)$  
            &  $25.14 (\pm 3.99)$ & $0.81 (\pm 0.10)$  \\[0.25cm]
        & {\eqref{e:maininverse-leastsquares}} & {$\monopsym$} 
            &  $25.33 (\pm 3.61)$ &  $0.81 (\pm 0.07)$ 
            &  $\best{26.09 (\pm 4.02)}$ &  $\best{0.83 (\pm 0.07)}$ \\[0.25cm]
        & {\eqref{e:maininverse-leastsquares}}  & {$\linopsym$} 
            & $18.77 (\pm 2.71)$ & $0.77 (\pm 0.12)$  
            & $18.77 (\pm 2.71)$ & $0.77 (\pm 0.12)$ \\[0.2cm] 
    
    \midrule
    \parbox[t]{2mm}{\multirow{3}{*}{\rotatebox[origin=c]{90}{\scriptsize$(K, \delta)=(5, 0.6)$}}}
        & {\eqref{e:maininverse}} & {$\monop$}
            &  $26.43 (\pm 4.23)$ & $\best{0.84 (\pm 0.09)}$ 
            &  $\best{26.63 (\pm 4.32)}$ & $\best{0.84 (\pm 0.09)}$ \\[0.25cm]
        & {\eqref{e:maininverse-leastsquares}} & {$\monopsym$} 
            &  $24.75 (\pm 4.33)$ &  $0.77 (\pm 0.13)$ 
            & $24.73 (\pm 4.32)$ & $0.77 (\pm 0.13)$ \\[0.25cm] 
        & {\eqref{e:maininverse-leastsquares}}  & {$\linopsym$}
            & $18.40 (\pm 2.74)$ & $0.72 (\pm 0.14)$ 
            & $18.40 (\pm 2.74)$ & $0.72 (\pm 0.14)$ \\[0.2cm] 
    \midrule
    \end{tabular}

    \caption{\small Results for low noise level $\sigma=0.01$: Average PSNR values (and standard deviation), obtained over $10$ images of BSD68, for solving the original inverse problem~\eqref{e:generalinverseproblem}, with $K \in \{1, 5\}$ and $\delta \in \{ 1, 0.6 \}$. 
    Results are shown when solving~\eqref{e:maininverse} with $\monop$,~\eqref{e:maininverse-leastsquares} with $\monopsym$, and~\eqref{e:maininverse-leastsquares} with $\linopsym$. PSNR and SSIM values correspond to the best results obtained when optimizing the regularization parameter~$\varrho$.
    }
    \label{tab:restoration_results-psnr-01}
\end{table}

\begin{table}[!t]
    \centering
    \setlength{\tabcolsep}{0.15cm}
    \small
    \begin{tabular}{l c l|r r | r r}
    \midrule
        & \multirow{2}{*}{Problem} & \multirow{2}{*}{Operator} & \multicolumn{2}{c|}{$\sigma_{\rm train}=0$} & \multicolumn{2}{c}{$\sigma_{\rm train}=0.01$} \\[0.1cm]
        &&& PSNR & SSIM & PSNR & SSIM  \\
    \midrule
    \parbox[t]{2mm}{\multirow{3}{*}{\rotatebox[origin=c]{90}{\scriptsize$(K, \delta)=(1, 1)$}}}
        & {\eqref{e:maininverse}} & {$\monop$} 
            &  $23.99 (\pm 3.54)$ &  $ 0.76 (\pm 0.13)$   
            &  $24.09 (\pm 3.56)$ &  $0.76 (\pm 0.13)$ \\[0.2cm]
        & {\eqref{e:maininverse-leastsquares}}& {$\monopsym$} 
            &  $24.95 (\pm 3.29)$ &  $0.78 (\pm 0.08)$ 
            &  $\best{25.44 (\pm 3.57)}$ &  $\best{0.80 (\pm 0.08)}$ \\[0.2cm]
        & {\eqref{e:maininverse-leastsquares}}  & {$\linopsym$} 
            &  $23.56 (\pm 3.04)$  &  $0.79 (\pm 0.11)$  
            & $23.60 (\pm 3.07)$  &  $0.79 (\pm 0.11)$ \\[0.2cm] 
    
    \midrule
    \parbox[t]{2mm}{\multirow{3}{*}{\rotatebox[origin=c]{90}{\scriptsize$(K, \delta)=(5, 1)$}}}
        & {\eqref{e:maininverse}}& {$\monop$} 
            & $24.28 (\pm 3.78)$ & $0.77 (\pm 0.13)$ 
            & $24.31 (\pm 3.80)$ & $0.77 (\pm 0.13)$ \\[0.2cm]
        & {\eqref{e:maininverse-leastsquares}}& {$\monopsym$} 
            &  $25.70 (\pm 4.17)$ &  $0.81 (\pm 0.10)$ 
            &  $\best{25.72 (\pm 4.21)}$ &  $\best{0.81 (\pm 0.11)}$ \\[0.2cm]
        & {\eqref{e:maininverse-leastsquares}}  & {$\linopsym$} 
            & $23.44 (\pm 3.33)$  &  $0.78 (\pm 0.12)$ 
            & $23.44 (\pm 3.33)$ & $0.78 (\pm 0.12)$ \\[0.2cm] 
    
    \midrule
    \parbox[t]{2mm}{\multirow{3}{*}{\rotatebox[origin=c]{90}{\scriptsize$(K, \delta)=(1, 0.6)$}}}
        &   {\eqref{e:maininverse}} & {$\monop$} 
            & $21.67 (\pm 2.89)$ & $0.73 (\pm 0.15)$  
            & $21.75 (\pm 2.94)$ & $0.73 (\pm 0.15)$   \\[0.25cm]
        & {\eqref{e:maininverse-leastsquares}} & {$\monopsym$} 
            &  $24.52 (\pm 3.68)$ &  $0.77 (\pm 0.11)$ 
            &  $\best{24.60 (\pm 3.72)}$ &  $\best{0.77 (\pm 0.11)}$ \\[0.25cm]
        & {\eqref{e:maininverse-leastsquares}}  & {$\linopsym$}
            & $18.51 (\pm 2.65)$ & $0.72 (\pm 0.14)$ 
            & $18.52 (\pm 2.66)$ & $0.72 (\pm 0.14)$ \\[0.2cm] 
    
    \midrule
    \parbox[t]{2mm}{\multirow{3}{*}{\rotatebox[origin=c]{90}{\scriptsize$(K, \delta)=(5, 0.6)$}}}
        & {\eqref{e:maininverse}} & {$\monop$}
            &  $21.72 (\pm 3.03)$ & $0.73 (\pm 0.15)$ 
            &  $21.76 (\pm 3.06)$ & $0.73 (\pm 0.15)$ \\[0.25cm]
        & {\eqref{e:maininverse-leastsquares}} & {$\monopsym$} 
            &  $24.29 (\pm 3.99)$ &  $\best{0.76 (\pm 0.12)}$ 
            & $\best{24.30 (\pm 4.00)}$ & $\best{0.76 (\pm 0.12)}$ \\[0.25cm] 
        & {\eqref{e:maininverse-leastsquares}}  & {$\linopsym$}
            & $18.36 (\pm 2.71)$ & $0.70 (\pm 0.16)$ 
            & $18.36 (\pm 2.71)$ & $0.70 (\pm 0.16)$  \\[0.2cm] 
    \midrule
    \end{tabular}

    \caption{\small 
    Results for high noise level $\sigma=0.05$: Average PSNR values (and standard deviation), obtained over $10$ images of BSD68, for solving the original inverse problem~\eqref{e:generalinverseproblem}, with $K \in \{1, 5\}$ and $\delta \in \{ 1, 0.6 \}$. 
    Results are shown when solving~\eqref{e:maininverse} with $\monop$,~\eqref{e:maininverse-leastsquares} with $\monopsym$, and~\eqref{e:maininverse-leastsquares} with $\linopsym$. PSNR and SSIM values correspond to the best results obtained when optimizing the regularization parameter~$\varrho$.
    }
    \label{tab:restoration_results-psnr-05}
\end{table}

\subsection{Restoration results}

In this section, we consider the original inverse problem~\eqref{e:generalinverseproblem}, with model~\eqref{e:generalF} for four cases: $(K, \delta) \in \{ (1, 1), (5, 1) , (1, 0.6), (5, 0.6) \}$. We further consider two noise levels on the inverse problem: $\sigma=0.01$ (low noise level) and $\sigma=0.05$ (high noise level).
Simulations are run on $10$ images from the BSD68 dataset. 
We compare the three methods described in Section~\ref{sssec:sim:image:mono} that hold convergence guarantees.
Precisely, we solve \eqref{e:maininverse} with $\monop$, \eqref{e:maininverse-leastsquares} with $\monopsym$, and \eqref{e:maininverse-leastsquares} with $\linopsym$. 
For the regularization $r$, we choose a smoothed Total Variation term \cite{getreuer_rudin-osher-fatemi_2012} (See Section~\ref{sec:smooth-tv} for details),
and we manually choose the regularization parameter $\varrho$ to achieve the best reconstruction quality for each method.

Quantitative results with average PSNR and SSIM values obtained for each method, in each setting, are reported in Tables \ref{tab:restoration_results-psnr-01} and \ref{tab:restoration_results-psnr-05}, for $\sigma=0.01$ and $\sigma=0.05$, respectively. 
We can observe overall  that the proposed least-squares approach $\monopsym$ always outperforms its linear counterpart $\linopsym$. It also performs better than $\monop$ for $K=1$, and similarly for $K=5$.
Regarding the noise level for training $\monop$ and $\monopsym$, either choice $\sigma_{\rm train}=0$ or
$\sigma_{\rm train}=0.01$ lead to very similar reconstruction results. Hence, the learned monotone approximation of $F$ does not seem affected by the noise level on the training dataset.

Qualitative results are presented in Figures \ref{fig:restoration_example_delta=1} and \ref{fig:restoration_example_delta=0.6} for the low noise level ($\sigma=0.01$), $K=5$ with $\delta = 1$ and $\delta=0.6$, respectively. 
Each figure shows results for two images, obtained by selecting the regularization parameter $\varrho$ leading to the best reconstruction quality. 
Observations for these two images validate the quantitative results reported in Table~\ref{tab:restoration_results-psnr-01}. 
The linear least squares procedure using $\Flinsym$ leads to the worst performance due to method failing to correct for the saturation function. 
For the high saturation case $\delta=0.6$, we see that $\monop$ leads to slightly sharper reconstructions, while both least-squares versions $\monopsym$ and $\linopsym$ seem to over-smooth the reconstruction, possibly due to the bias introduced with respect to the true saturation model.

We also present the convergence profiles $\| x_{k+1} - x_k \|_2 / \|y\|_2$ with respect to the iterations $k$ of the restored images using the three considered models $\monop$, $\monopsym$ and $\linopsym$, with two training noise levels $\sigma_{\rm train}\in \{0, 0.01\}$. We observe that all methods exhibit a converging behaviour.  
The oscillations are due to the Goldstein-Armijo rule that introduces a backtracking line search procedure in the algorithms to find the optimal step size. Smoother curves can be obtained by lowering the $\beta$ parameters in~\eqref{def:armijo}, with the caveat of more processing time due to the additional steps performed within the backtracking line search procedure.

\begin{figure}[ht]
    \scalebox{1}{
    \centering
    \scriptsize
    \setlength{\tabcolsep}{0.01cm}
    \begin{tabular}{cccc}\zoomedinimagebis{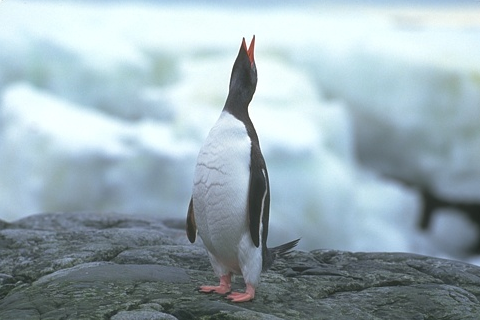}{210px 200px 210px 25px}
        & \zoomedinimagebis{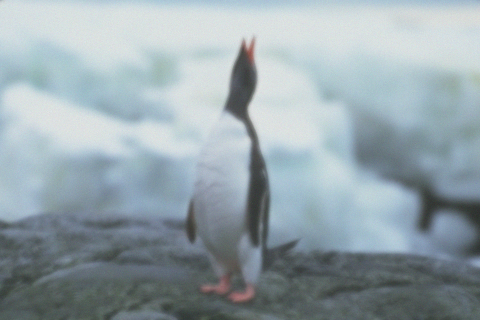}{210px 200px 210px 25px}
        &   \zoomedinimagebis{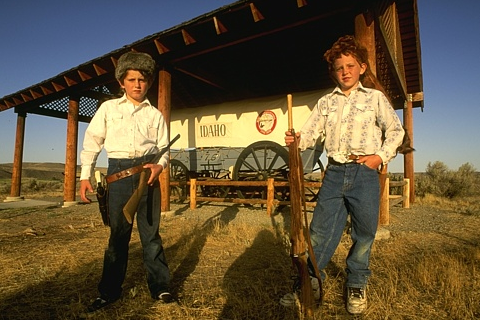}{193px 155px 243px 115px}
        &   \zoomedinimagebis{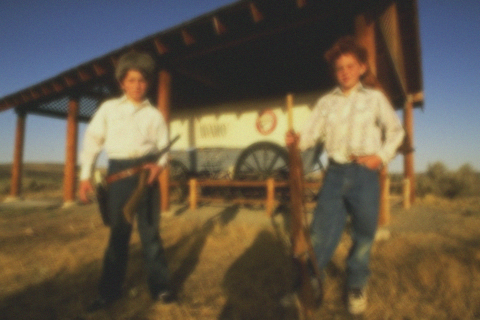}{193px 155px 243px 115px} \\
        $\overline{x}$ & $y$ -- $(23.96, 0.89)$   
        &   $\overline{x}$ & $y$ -- $(21.11 , 0.65)$   \\[0.1cm]

        \zoomedinimagebis{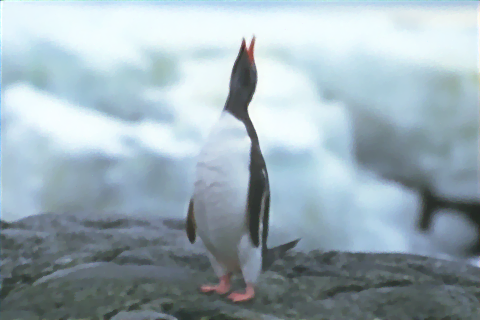}{210px 200px 210px 25px}
        &   \zoomedinimagebis{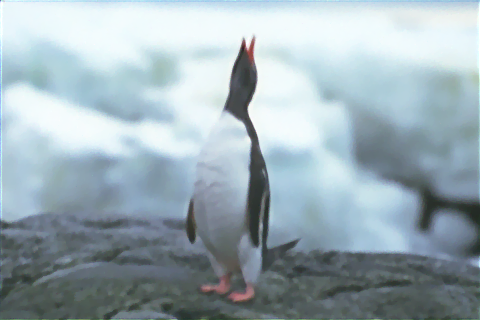}{210px 200px 210px 25px}  
        &   \zoomedinimagebis{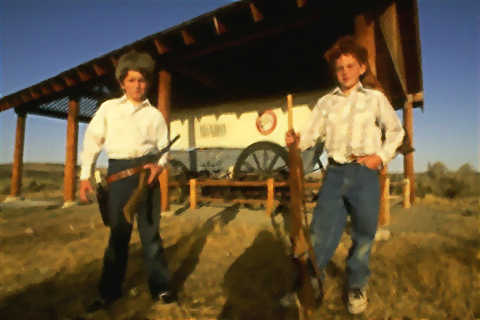}{193px 155px 243px 115px}
        &   \zoomedinimagebis{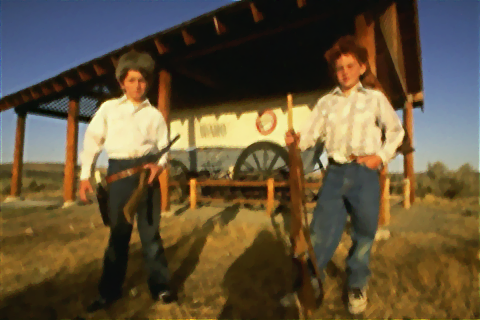}{193px 155px 243px 115px}  \\
        $\widehat{x}_{\monop}$ -- $(31.59, 0.93)$ & $\widehat{x}_{\monop}$ -- $(30.98, 0.93)$ 
        &   $\widehat{x}_{\monop}$ -- $(24.95,0.80)$ & $\widehat{x}_{\monop}$ -- $(24.67, 0.80)$ \\
        $\sigma_{\rm train}=0$ & $\sigma_{\rm train}=0.01$
        &   $\sigma_{\rm train}=0$ & $\sigma_{\rm train}=0.01$ \\[0.1cm]

        \zoomedinimagebis{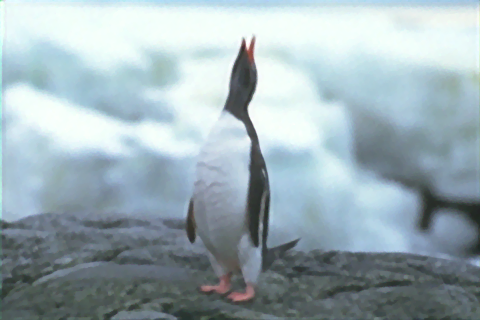}{210px 200px 210px 25px}
        &   \zoomedinimagebis{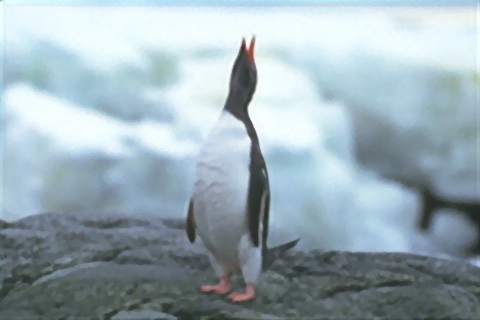}{210px 200px 210px 25px}  
        &   \zoomedinimagebis{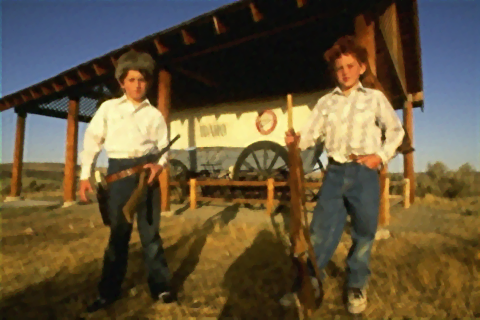}{193px 155px 243px 115px}
        &   \zoomedinimagebis{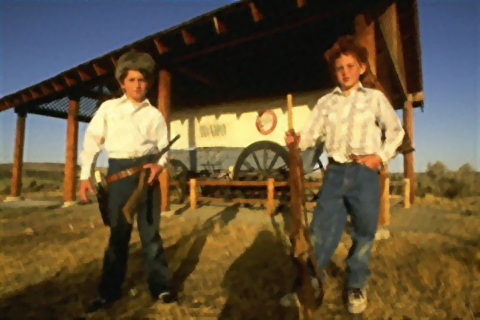}{193px 155px 243px 115px}  \\
        $\widehat{x}_{\monopsym}$ -- $(31.84, 0.94)$ & $\widehat{x}_{\monopsym}$ -- $(32.10, 0.94)$ 
        &   $\widehat{x}_{\monopsym}$ -- $(25.75,0.83)$ & $\widehat{x}_{\monopsym}$ -- $(25.72, 0.83)$ \\
        $\sigma_{\rm train}=0$
        & $\sigma_{\rm train}=0.01$
        & $\sigma_{\rm train}=0$
        & $\sigma_{\rm train}=0.01$ \\[0.1cm]

        \zoomedinimagebis{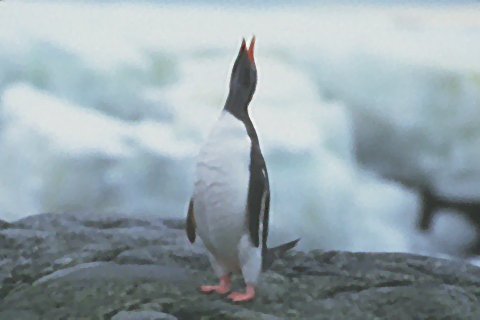}{210px 200px 210px 25px}
        &   \zoomedinimagebis{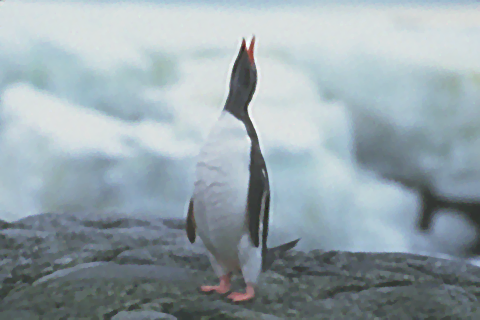}{210px 200px 210px 25px}  
        &   \zoomedinimagebis{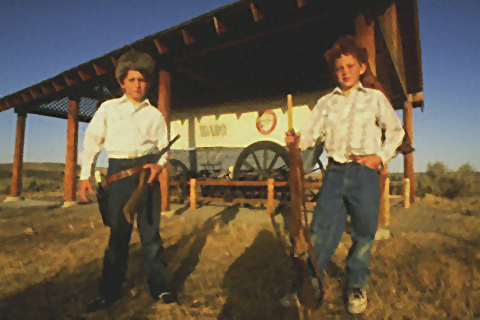}{193px 155px 243px 115px}
        &   \zoomedinimagebis{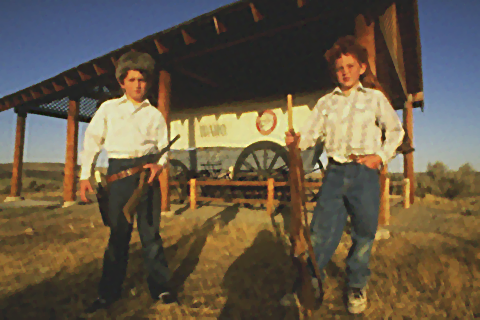}{193px 155px 243px 115px}  \\
        $\widehat{x}_{\linopsym}$ -- $(24.67, 0.93)$ 
        & $\widehat{x}_{\linopsym}$ -- $(24.69, 0.93)$ 
        & $\widehat{x}_{\linopsym}$ -- $(22.77,0.75)$ 
        & $\widehat{x}_{\linopsym}$ -- $(22.79, 0.75)$ \\
        $\sigma_{\rm train}=0$
        & $\sigma_{\rm train}=0.01$
        & $\sigma_{\rm train}=0$
        & $\sigma_{\rm train}=0.01$ \\[0.3cm]

        
        \multicolumn{2}{c}{\includegraphics[width=0.49\textwidth]{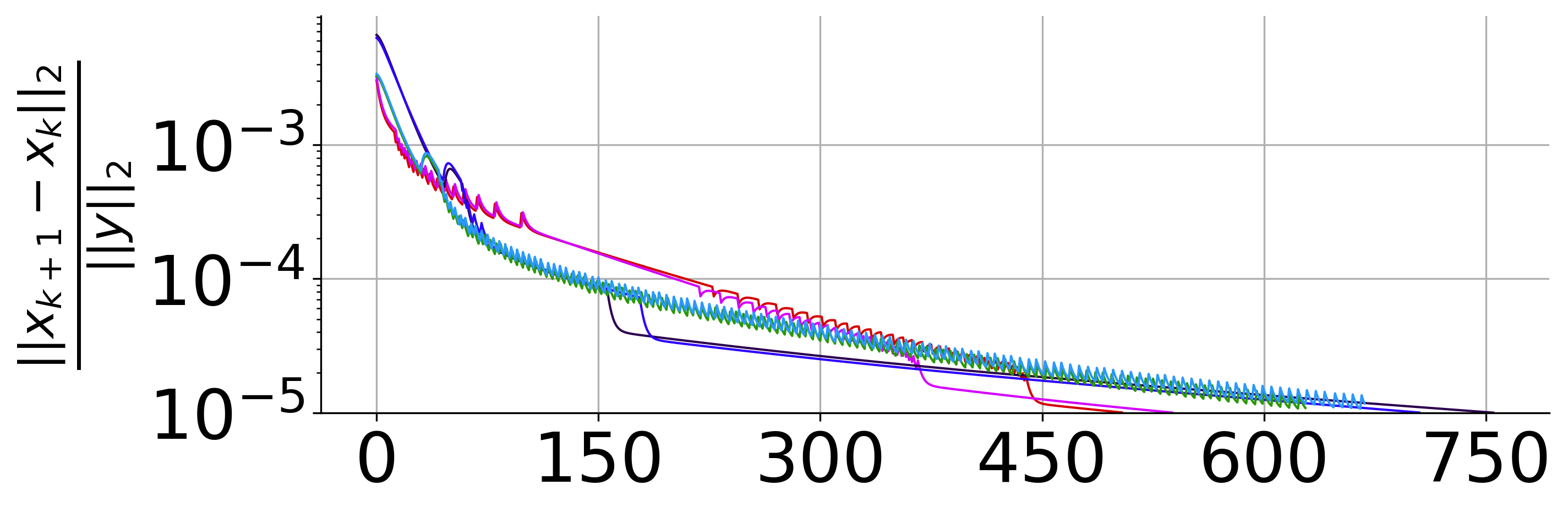}}
        &   \multicolumn{2}{c}{\includegraphics[width=0.49\textwidth]{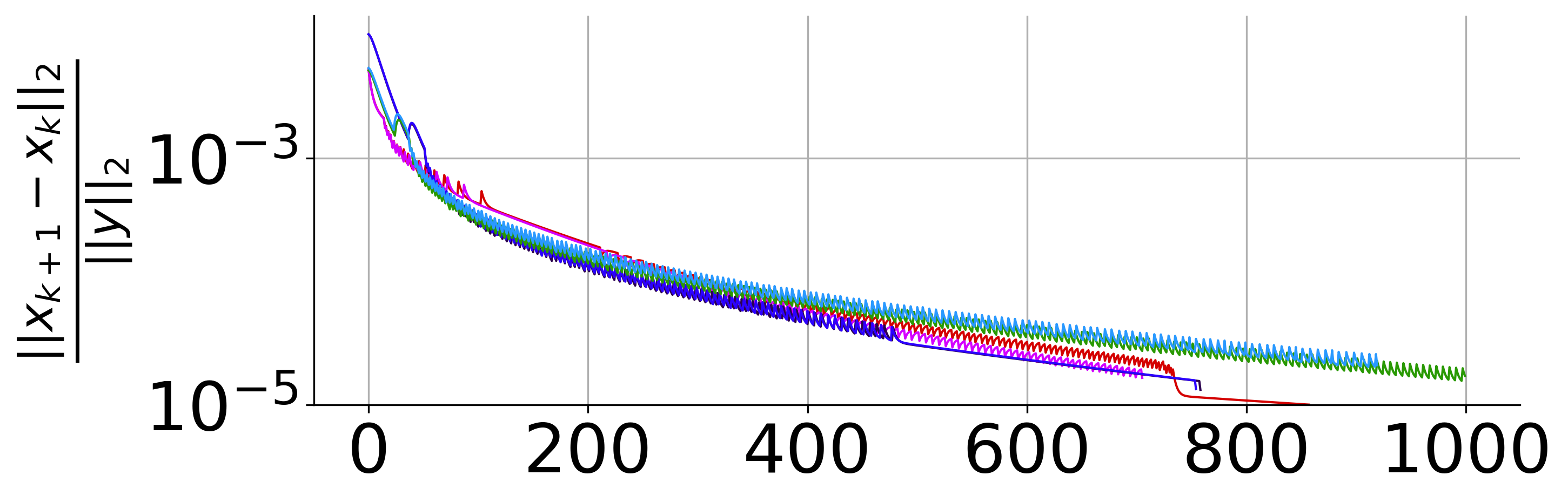}} \\ 
        \multicolumn{4}{c}{\includegraphics[width=0.6\textwidth]{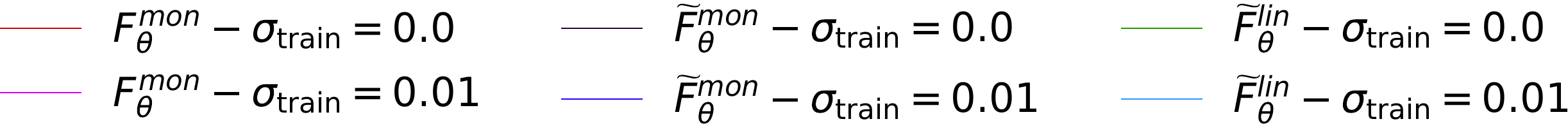}}
    \end{tabular}
    }

    \centering
    \caption{\small
    Results for low noise level $\sigma=0.01$: Restoration results for $K=5$ and $\delta = 1$, for two images. For each image and method, we provide (PSNR, SSIM) values between the solution and $\overline{x}$. Last row shows the convergence profiles associated with the reconstruction of each image, for the three considered models.
    }
    \label{fig:restoration_example_delta=1}
\end{figure}

\begin{figure}[!ht]
    \centering
    \scriptsize
    \setlength{\tabcolsep}{0.01cm}
    \begin{tabular}{cccc}
        \zoomedinimagebis{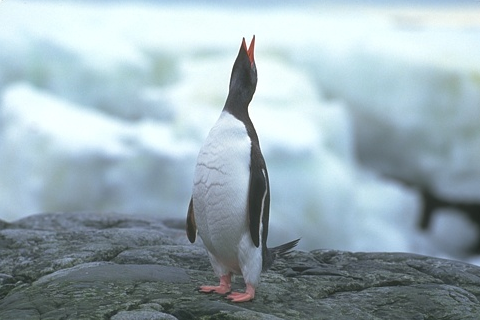}{210px 200px 210px 25px} 
        & \zoomedinimagebis{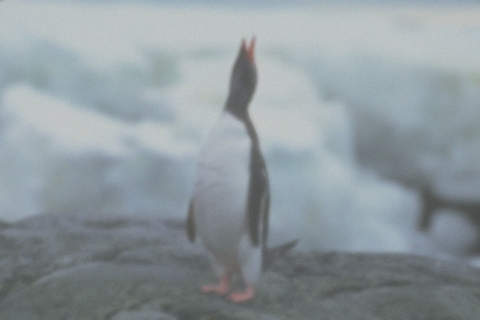}{210px 200px 210px 25px} 
        & \zoomedinimagebis{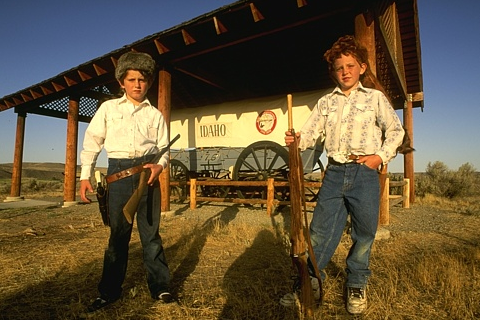}{193px 155px 243px 115px}
        & \zoomedinimagebis{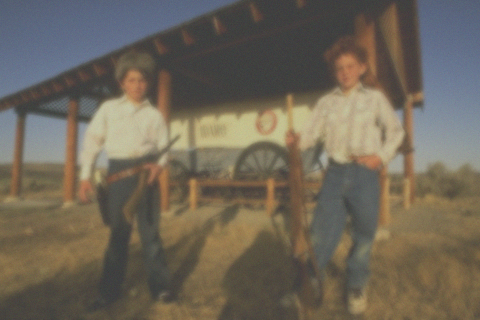}{193px 155px 243px 115px} \\
        $\overline{x}$ 
        & $y$ -- $(17.30, 0.85)$ 
        & $\overline{x}$ 
        &  $y$ -- $(16.27, 0.54)$ \\

        \zoomedinimagebis{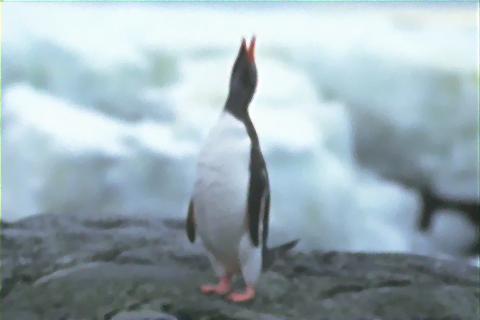}{210px 200px 210px 25px} 
        & \zoomedinimagebis{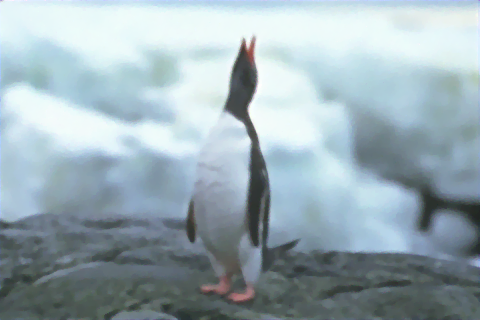}{210px 200px 210px 25px}
        & \zoomedinimagebis{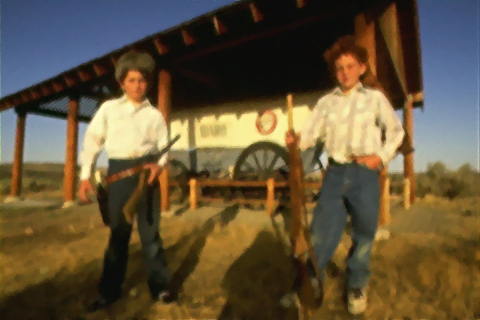}{193px 155px 243px 115px} 
        & \zoomedinimagebis{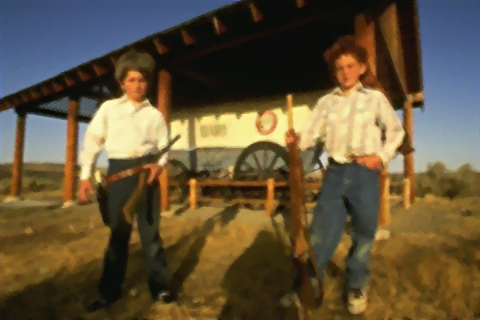}{193px 155px 243px 115px} \\ 
        $\widehat{x}_{\monop}$ -- $(30.64, 0.92)$ 
        & $\widehat{x}_{\monop}$ -- $(30.59, 0.92)$
        & $\widehat{x}_{\monop}$ -- $(24.12, 0.77)$ 
        & $\widehat{x}_{\monop}$ -- $(24.21, 0.78)$ \\ 
        $\sigma_{\rm train}=0$
        & $\sigma_{\rm train}=0.01$
        & $\sigma_{\rm train}=0$
        & $\sigma_{\rm train}=0.01$ \\[0.1cm]
        
        \zoomedinimagebis{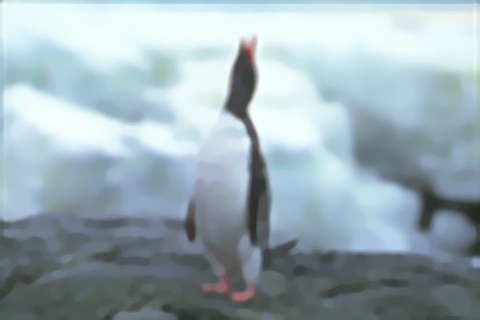}{210px 200px 210px 25px}
        &   \zoomedinimagebis{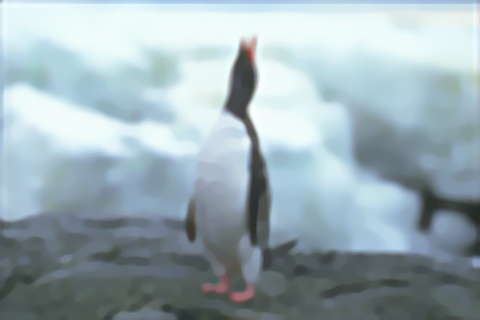}{210px 200px 210px 25px}  
        &   \zoomedinimagebis{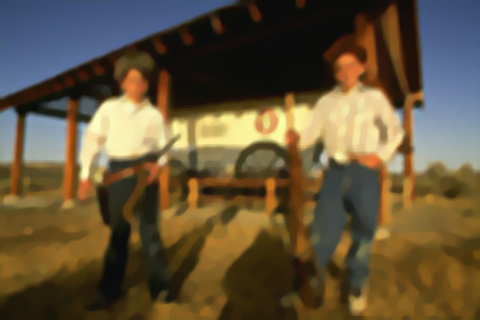}{193px 155px 243px 115px}
        &  \zoomedinimagebis{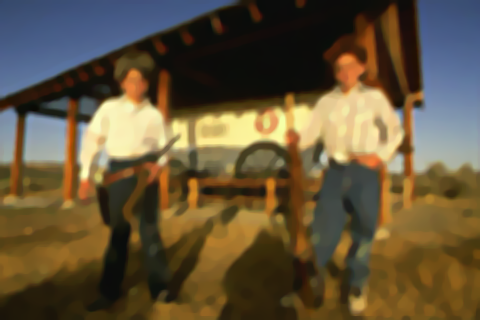}{193px 155px 243px 115px}  \\
        $\widehat{x}_{\monopsym}$ -- $(28.91, 0.89)$ & $\widehat{x}_{\monopsym}$ -- $(28.87, 0.89)$ 
        &   $\widehat{x}_{\monopsym}$ -- $(22.55,0.7)$ & $\widehat{x}_{\monopsym}$ -- $(22.54, 0.7)$ \\
        $\sigma_{\rm train}=0$
        & $\sigma_{\rm train}=0.01$
        & $\sigma_{\rm train}=0$
        & $\sigma_{\rm train}=0.01$ \\[0.1cm]

        \zoomedinimagebis{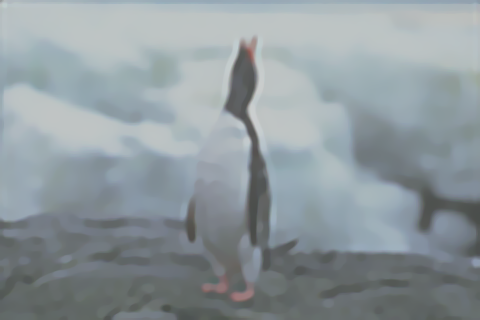}{210px 200px 210px 25px}
        &   \zoomedinimagebis{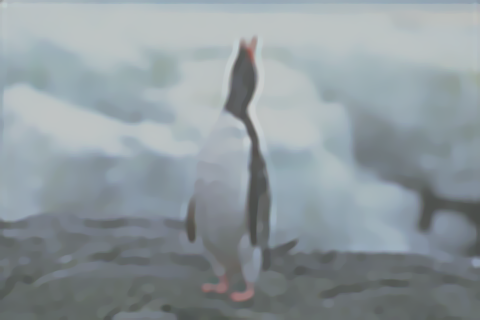}{210px 200px 210px 25px}  
        &   \zoomedinimagebis{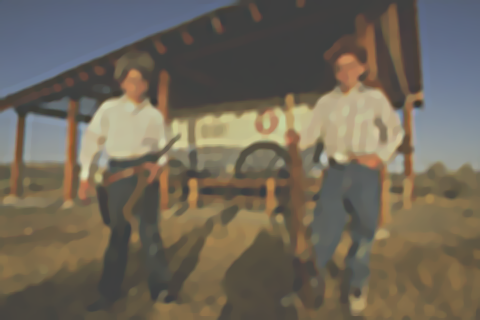}{193px 155px 243px 115px}
        &   \zoomedinimagebis{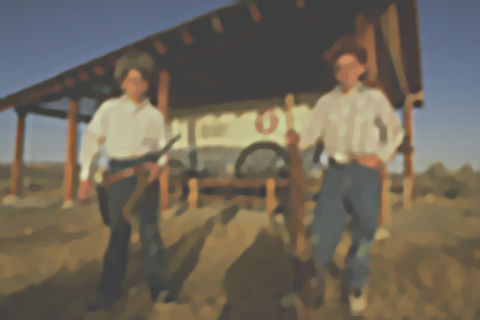}{193px 155px 243px 115px}  \\
        $\widehat{x}_{\linopsym}$ -- $(17.40, 0.86)$ 
        & $\widehat{x}_{\linopsym}$ -- $(17.40, 0.86)$ 
        & $\widehat{x}_{\linopsym}$ -- $(16.42, 0.54)$ 
        & $\widehat{x}_{\linopsym}$ -- $(16.41, 0.54)$ \\
        $\sigma_{\rm train}=0$
        & $\sigma_{\rm train}=0.01$
        & $\sigma_{\rm train}=0$
        & $\sigma_{\rm train}=0.01$ \\[0.3cm]

        \multicolumn{2}{c}{\includegraphics[width=0.49\textwidth]{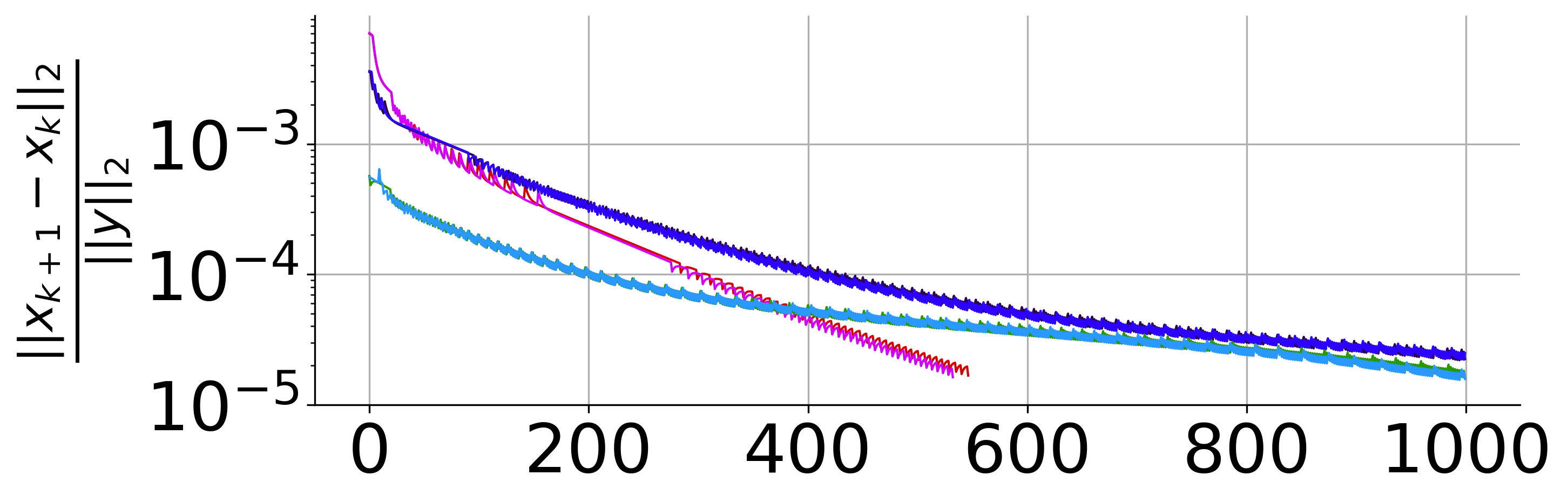}}
        &   \multicolumn{2}{c}{\includegraphics[width=0.49\textwidth]{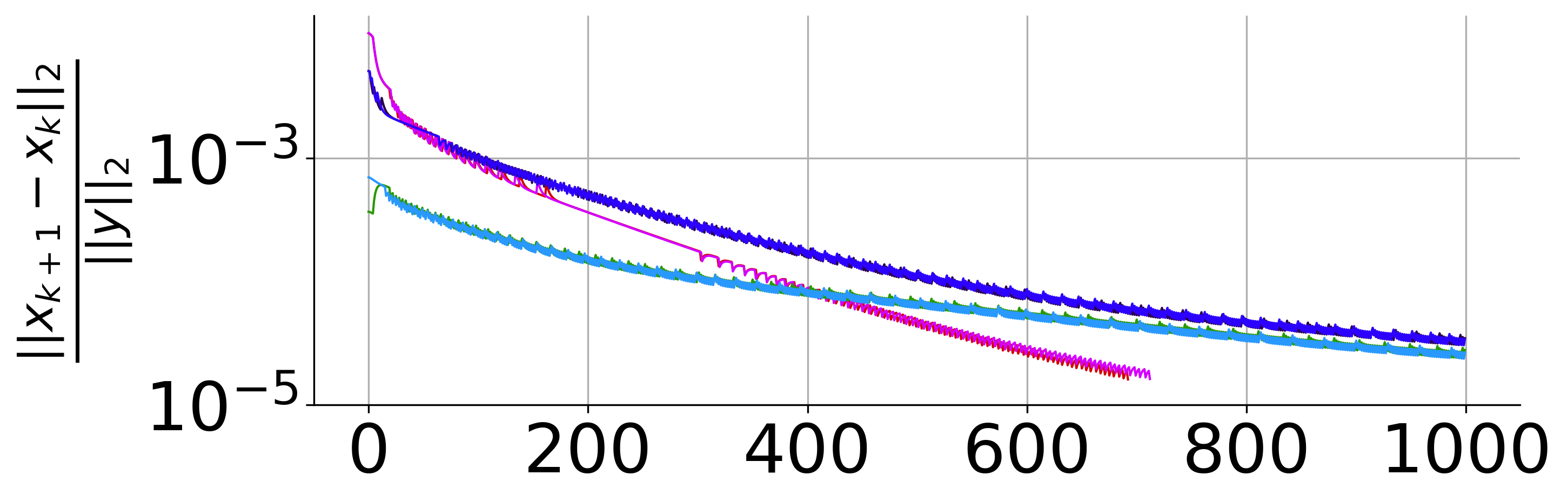}} \\ 
        \multicolumn{4}{c}{\includegraphics[width=0.49\textwidth]{Figures/Simulation/Results20240201/noisy/legend.pdf}}
    \end{tabular}

    \centering
    \caption{\small 
    Results for low noise level $\sigma=0.01$: 
    Restoration results for $K=5$ and $\delta = 0.6$, for two images. For each image and method, we provide (PSNR, SSIM) values between the solution and $\overline{x}$. Last row shows the convergence profiles associated with the reconstruction of each image, for the three considered models.
    }
    \label{fig:restoration_example_delta=0.6}
\end{figure}


\section{Conclusions}
\label{Sec:conclusions}

In this article, we introduced a novel approach for training monotone neural networks, by designing a penalized training procedure. The resulting monotone networks can then be embedded in the FBF algorithm, within a PnP framework, yet ensuring the convergence of the iterates of the resulting algorithm. This method can be leveraged for addressing a wide range of monotone inclusion problems, including in imaging, as demonstrated in the context of solving non-linear inverse imaging problems.

The proposed PnP-FBF method enables solving generic constrained monotone inclusion problem. 
Its convergence is ensured even if the involved operators are not cocoercive. Hence, the proposed method can be used for a wider class of operators compared to other similar iterative schemes such as the (PnP) forward-backward method. Moreover, combined with the Armijo-Goldstein rule, the proposed FBF-PnP algorithm is guaranteed to converge without needing an explicit computation of the Lipschitz constant of the neural network.

The proposed penalized training procedure enables us to learn monotone neural networks. To do so, we designed a differentiable penalization function, relying on properties of the network Jacobian, that can be implemented efficiently using auto-differentiation tools. The proposed training approach is very flexible and can be applied to a wide range of network architectures and training paradigms. 

We finally illustrated the benefit of the proposed framework in learning monotone operators for semi-blind non-linear deconvolution imaging problems. Our methodology demonstrated an accurate monotone approximation of the true non-monotone degradation function. We show that the monotonocity of the learned network is further instrumental within the proposed PnP-FBF scheme for inverting the network, and solving the image recovery problem. 

The versatility of our methodology allows its potential extension to various imaging problems involving a nonlinear model, including super-resolution.
The practical utility of 
monotone operators has already been explored in normalizing flows, offering possibilities for extending our work to monotone normalizing flows \cite{ahn_invertible_2022}. Further investigations into style transfer tasks, leveraging the invertibility property, could yield stable neural networks for image-to-image mapping problems. Finally, the application of non-linear inverse problems in tomography, acknowledging the inherent saturation in tissue absorption of X-rays, offers a promising avenue for future exploration.


\section*{Acknowledgments}
We would like to thank the DATAIA institute from University Paris-Saclay for funding the doctoral work of YB. The work of JCP was founded by the ANR Research and Teaching Chair in Artificial Intelligence, BRIDGEABLE. The work of AR was partly funded by the Royal Society of Edinburgh; EPSRC grant EP/X028860/1; the CVN, Inria/OPIS, and CentraleSup\'elec.


\bibliographystyle{plain}
\bibliography{abbr,references}


\appendix

\section{Saturation function}
\label{Sec:SaturationFunction}

We modify the classical Tanh function to be centered on 0.5 and to take as input images with pixel values in the range $[0, 1]$. 
We also introduce a scaling factor $\delta$ to accentuate the non-linearity of the tanh (Figure~\ref{figure:SaturationFigure}). 
The resulting function is then used for saturation, and is defined as 
\begin{equation}\label{e:SaturationUpsilonDef}
    S_\delta \colon \defspace \to  \defspace \colon
    \boldsymbol{x}=(x_i)_{1\le i \le n} \mapsto  (\psi_\delta(x_i))_{1 \le i \le n} 
\end{equation}
where
\begin{equation}
    \psi_\delta \colon \RR \to \RR \colon
    x \mapsto  \frac{\text{tanh}(\delta (2x - 1)) + 1}{2} .
\end{equation}
Examples for $\delta\in \{0.6, 1\}$ are shown in Figure~\ref{figure:SaturationFigure}.

\begin{figure}[h]
    \centering
    \includegraphics[width=0.5\linewidth]{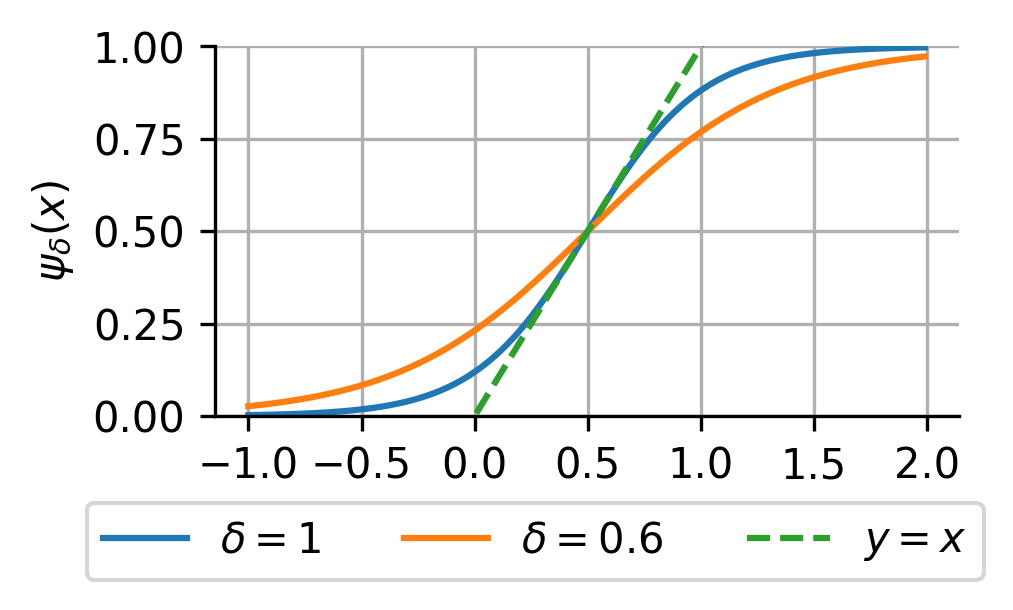}
    \caption{$\psi_\delta$ functions used in $S_\delta$ in the experiments, for $\delta \in \{0.6, 1\}$ (blue and orange, respectively). The green line represents the  identity function $f(x) = x$, highlighting the non-linearity of the saturation function.}
    \label{figure:SaturationFigure}
\end{figure}

\section{Linear approximation of a convolution filter}\label{sec:linear-approx}

As explained in Section~\ref{sssec:sim:image:mono}, we compare our model to a learned linear approximation $\linop$ of $\F$.  
To this aim, we define $\linop$ as a 2D convolution layer, with convolution kernel $f_\theta \in \mathbb R^D$. Then, we train $\linop$ by solving 
\begin{equation}\label{e:linear-approx-objective-repara}
    \argmin_{\theta \in \RR^D} \Sum_{(x, y) \in \mathbb{D}_{\rm train}}\| \upsilon(f_\theta) \ast x - y \|_1,
\end{equation}
where 
\begin{equation}
    \displaystyle \upsilon \colon \mathbb R^D \to \mathbb R^D
    \colon f_\theta \mapsto 
    \frac{\text{ReLU}(f_\theta)}{\sum_{i \in 1}^D \max (0, f_{\theta_i})}.
\end{equation}
This parametrization is introduced as we assume that the true convolution kernels in~\eqref{e:generalF} are normalized and nonnegative.

Problem~\eqref{e:linear-approx-objective-repara} is then solved using Adam optimizer for 200 epochs, with a learning rate of $0.02$. The size of the kernels was chosen equal to the size of the true kernels in~\eqref{e:generalF} for simplicity (i.e., $D=9 \times 9$). The learned kernels are displayed in Figure~\ref{figure:learned-kernels}. 
We can observe that for the normal saturation levels ($\delta = 1$), the approximated kernels are almost equal to the true ones, with or without noise. For the high saturation level ($\delta = 0.6$) however, the learned kernels are very different from the true ones, certainly to try to compensate for the nonlinear saturation function.

\begin{figure}[!h]
    \centering\footnotesize
    \setlength{\tabcolsep}{0.1cm}
    \begin{tabular}{cccccc}
    & \multirow{2}{*}{True kernel} 
    & \multicolumn{2}{c}{$\delta = 1$}
    & \multicolumn{2}{c}{$\delta = 0.6$} \\ 

    &  
    & $\sigma_{\rm train} = 0$
    & $\sigma_{\rm train} = 0.01$
    & $\sigma_{\rm train} = 0$
    & $\sigma_{\rm train} = 0.01$ \\ 
    
    \rotatebox[origin=lB]{90}{\quad\quad$K=1$}
    & \includegraphics[width=0.15\linewidth]{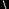}
    &   \includegraphics[width=0.15\linewidth]{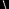}
    &   \includegraphics[width=0.15\linewidth]{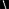}
    &   \includegraphics[width=0.15\linewidth]{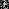}
    &   \includegraphics[width=0.15\linewidth]{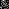} \\

    \rotatebox[origin=lB]{90}{\quad\quad$K=5$}
    & \includegraphics[width=0.15\linewidth]{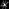}
    &   \includegraphics[width=0.15\linewidth]{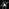}
    &   \includegraphics[width=0.15\linewidth]{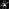}
    &   \includegraphics[width=0.15\linewidth]{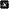}
    &   \includegraphics[width=0.15\linewidth]{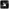} \\
    \end{tabular}

    \vspace*{-0.2cm}
    \caption{\small Learned normalized kernels $\upsilon(f_\theta)$ for the different settings considered in Section~\ref{Sec:simulations}.}%
    \label{figure:learned-kernels}
\end{figure}

\section{Total variation regularization}\label{sec:smooth-tv}

We provide here the details of the smoothed approximation used in the manuscript for the Total Variation \cite{getreuer_rudin-osher-fatemi_2012} function. 
We define
\begin{equation}
    (\forall x \in \defspace)\quad
    r(x) = \sum_{i = 1}^n \sqrt{[\nabla_h x]_i^2 + [\nabla_v x]_i^2 + \varepsilon}
\end{equation}
where $\nabla_h \colon \defspace \to \defspace$ and $\nabla_v \colon \defspace \to \defspace$ model the linear horizontal and vertical gradient operators, respectively, and $\varepsilon>0$ is a smoothing parameter. Then, the gradient of $r$ is given by
\begin{equation}
    \displaystyle (\forall x \in \defspace)\quad
    \nabla r(x) = \left( \dfrac{ \Big[ \nabla_h^\top (\nabla_h x) + \nabla_v^\top (\nabla_v x) \Big]_i}{\sqrt{[\nabla_h x]_i^2 + [\nabla_v x]_i^2 + \varepsilon}} \right)_{1 \le i \le n}.
\end{equation}

\end{document}